\documentclass[12pt]{article}
\usepackage{amsthm}
\usepackage{amsmath}
\usepackage{amsfonts}
\usepackage{graphicx}
\usepackage{latexsym}
\usepackage{amssymb}

\DeclareMathAlphabet{\mathbfsl}{OT1}{ppl}{b}{it} %{OT1}{cmr}{bx}{it}

\newcommand{\deff}{\mbox{$\stackrel{\rm def}{=}$}}
\newcommand{\Span}[1]{{\left\langle {#1} \right\rangle}}

\makeatletter
 \DeclareRobustCommand{\nsbinom}{\genfrac[]\z@{}}
 \makeatother
 
 \newcommand{\sbinomq}[2]{\nsbinom{{#1}}{{#2}}_{q}}

\newcommand{\field}[1]{\mathbb{#1}}

\newcommand{\F}{\field{F}}

\newcommand{\dS}{\field{S}}

\newcommand{\cA}{{\cal A}}
\newcommand{\cB}{{\cal B}}
\newcommand{\cC}{{\cal C}}
\newcommand{\cD}{{\cal D}}

\newcommand{\cP}{{\cal P}}

\newcommand{\cY}{{\cal Y}}
\newcommand{\cZ}{{\cal Z}}
\newcommand{\cM}{{\cal M}}

\newcommand{\cG}{{\cal G}}

\newcommand{\linadd}{\kern1pt\mbox{\small$\boxplus$}\kern1pt}

\newtheorem{theorem}{Theorem}
\newtheorem{lemma}{Lemma}

\newtheorem{cor}{Corollary}

\begin{document}

\bibliographystyle{plain}

\title{
\begin{center}
Residual $q$-Fano Planes and Related Structures
\end{center}
}
\author{
{\sc Tuvi Etzion}\thanks{Department of Computer Science,
Technion, Haifa 32000, Israel, e-mail: {\tt
etzion@cs.technion.ac.il}.} \and {\sc Niv
Hooker}\thanks{Department of Computer Science, Technion, Haifa
32000, Israel, e-mail: {\tt snivh@cs.technion.ac.il}.}}

\maketitle

\begin{abstract}
One of the most intriguing problems, in $q$-analogs of designs, is the existence question of
an infinite family of $q$-analog of Steiner systems, known also as $q$-Steiner systems, (spreads not included) in general,
and the existence question for the $q$-analog of the Fano plane, known also as the $q$-Fano plane, in particular. These
questions are in the front line of open problems in block design.
There was a common belief and a conjecture that such structures do not exist.
Only recently, $q$-Steiner systems were found for one set of parameters.
In this paper, a definition for the $q$-analog of the residual design
is presented. This new definition is different from previous known definition,
but its properties reflect better the $q$-analog properties.
The existence of a design with the parameters
of the residual $q$-Steiner system in general and the
residual $q$-Fano plane in particular are examined.
We construct different
residual $q$-Fano planes for all $q$, where $q$ is a prime power. The constructed structure
is just one step from a construction of a $q$-Fano plane.
\end{abstract}

\vspace{0.5cm}


\vspace{0.5cm}

%\noindent
%{\bf Mathematics Subject Classification}: 05B40, 51E10  .

%\footnotetext[1] { This research was supported in part by the Israeli
%Science Foundation (ISF), Jerusalem, Israel, under
%Grant 10/12.}

%%%%%%%%%%%%%%%%%%%%%%%%%%%%%%%%%%%%%%%%%%%%%%%%%%%%%%%%%%%%%%%%%%%%%%
%%%%%%%%%%%%%%%%%%%%%%%%%%%%%%%%%%%%%%%%%%%%%%%%%%%%%%%%%%%%%%%%%%%%%%
%%%%%%%%%%%%%%%%%%%%%%%%%%%%%%%%%%%%%%%%%%%%%%%%%%%%%%%%%%%%%%%%%%%%%%
\newpage
\section{Introduction}

Let $\F_q$ be the finite field with $q$ elements
and let $\F_q^n$ be the set of all vectors
of length $n$ over~$\F_q$. $\F_q^n$ is a vector
space with dimension $n$ over $\F_q$. For a given integer $k$,
$0 \leq k \leq n$, let $\cG_q(n,k)$ denote the set of all
$k$-dimensional subspaces ($k$-\emph{subspaces} in short) of $\F_q^n$. $\cG_q(n,k)$ is often
referred to as a Grassmannian. It is well known that
$$ \begin{small}
| \cG_q (n,k) | = \sbinomq{n}{k}
\deff \frac{(q^n-1)(q^{n-1}-1) \cdots
(q^{n-k+1}-1)}{(q^k-1)(q^{k-1}-1) \cdots (q-1)}
\end{small}
$$
where $\sbinomq{n}{k}$ is the $q$-\emph{binomial coefficient} (known also
as the \emph{Gaussian coefficient}~\cite[pp. 325-332]{vLWi92}).

Let $Q$ be a set with $n$ elements. A \emph{$t$-$(n,k,\lambda)$
design}, is a collection of $k$-subsets
of $V$, called \emph{blocks}, such that each $t$-subset of $Q$ is contained in
exactly $\lambda$ blocks.
A $t$-$(n,k,\lambda)$ design with $t = \lambda = 1$ is trivial: it is simply
a~partition of $Q$ into $k$-subsets, which exists if and only if $k$ divides $n$.
A $t$-$(n,k,1)$ design with $t \:{\geq}\, 2$
is known as a \emph{Steiner system}, and usually
denoted $S(t,k,n)$. Steiner systems are among the most
beautiful and well-studied structures in combinatorics. Their history
goes back to the work of Pl\"ucker~\cite{P1835},
Kirkman~\cite{K1847}, Cayley~\cite{C1850}, and
Steiner~\cite{S1853} in the first half of the 19-th century. %R the early 1850s.
%RR By now, At present
Today, the~significance of Steiner systems extends well beyond
combinatorics --- they have found applications in many areas,
including group theory, finite geometry, cryptography, and
coding theory~\cite{BJL,CoDi07,EtVa11}.
For example, a finite projective plane of order $q$ can be
characterized as a~Steiner system $S(2,q{+}1,q^2{+}\,q\,{+}1)$,
with lines as blocks. As another example,
the Mathieu groups (which played an~important
role in the classification of finite simple groups)
are most naturally understood as automorphism groups
of certain Steiner systems.

A long-standing problem in design theory asks whether nontrivial
(meaning $t < k < n$) Steiner systems with $t > 5$ exist. Keevash
recently announced~a~resolution of this problem: his breakthrough
paper~\cite{Kee14} moreover shows that Steiner~systems $S(t,k,n)$ exist
for all $t < k$ and all sufficiently large integers $n$ that satisfy
the necessary divisibility conditions. More recently
another (simpler) proof was provided by Glock, K\"{u}hn, Lo, and Osthus~\cite{GKLO16}.

The classical theory of \emph{q-analogs} of mathematical objects
and functions has its beginnings in the work of Euler~\cite{Euler,KvA09}.
In 1957, Tits~\cite{Tits57} further suggested that combinatorics of sets could
be regarded as the limiting case $q \to 1$ of combinatorics of vector spaces
over the finite field $\F_q$. Indeed, there is a strong analogy between subsets
of a set and subspaces of a vector space, expounded by numerous
authors---see~\cite{Cohn,GR,Wang} and references therein.
It is therefore natural~to~ask which combinatorial structures
can be generalized from sets (the $q \to 1$ case) to vector spaces
over $\F_q$. For $t$-designs and Steiner systems, this question
was first studied by Cameron~\cite{Cam74,Cam74a} and
Delsarte~\cite{Del76} in the early 1970s.
Specifically,~let $\F_q^n$ be a vector space of dimension $n$ over the
finite field $\F_q$.
Then a \emph{$t$-$(n,k,\lambda)$~design over $\F_q$} is defined
in~\cite{Cam74,Cam74a,Del76} as
a collection of $k$-subspaces of $\F_q^n$,
called \emph{blocks}, such that each $t$-subspace of $\F_q^n$
is contained in exactly $\lambda$ blocks. Such $t$-designs over $\F_q$
are the $q$-analogs of conventional combinatorial designs. By analogy
with the $q \to 1$ case, a~$t$-$(n,k,1)$ design over $\F_q$~is said to
be a \emph{$q$-Steiner system}, and denoted $\dS_q(t,k,n)$.

\vspace{1.00ex}
\noindent{\bf Remark.}
We observe that $q$-analogs of designs and Steiner systems are
not~only of interest in their own right, but also arise naturally
in other areas, such as network coding~\cite{EtSt16}.
The appropriate code in random network coding is
a~{collection of subspaces} of $\F_q^n$ that are well-separated
according to a~metric defined on the Grassmannian.
Consequently, a $q$-Steiner~system $\dS_q(t,k,n)$ can be thought
of as an \emph{optimal code} for error-correction in networks.
For more details on this, see~\cite{EtVa11,KoKs08}.

\vspace{1.00ex}
Following the work of Cameron~\cite{Cam74,Cam74a} and
Delsarte~\cite{Del76}, the first \mbox{examples}
of nontrivial $t$-designs over $\F_q$
were found by Thomas~\cite{Tho87} in 1987. Today,~\mbox{owing}
to the efforts of many
authors~\cite{BKL,FLV14,KiLa15,MMY95,RaSi89,Suz90,Suz90a,Suz92,Tho96},
numerous such examples are known.

However, the situation is very different for $q$-Steiner systems.
They are known to exist in the trivial cases $t = k$ or $k=n$,
and in the case where $t = 1$ and $k$~divides $n$.
In the latter case, $q$-Steiner systems coincide with the classical
notion of \emph{spreads} in projective geometry~\cite[Chapter\,24]{vLWi92}.
Some 40 years ago, Beutelspacher~\cite{Beu78} asked
whether nontrivial $q$-Steiner
systems with $t \geq 2$ exist, and this question has tantalized mathematicians
ever since. The problem has been studied by numerous
authors~\cite{AAK01,EtVa11a,Met99,ScEt02,Tho87,Tho96}, without much
progress toward constructing such $q$-Steiner systems.
In particular, Thomas~\cite{Tho96} showed in 1996 that
certain kinds of $\dS_2(2,3,7)$ $q$-Steiner systems (the smallest possible
example) cannot exist. %RR In 1999,
Three years later, Metsch~\cite{Met99} conjectured that~nontrivial
$q$-Steiner systems with $t \geq 2$ do not exist in general.
In contrast to this conjecture, a $q$-Steiner system $\dS_2 (2,3,13)$
was constructed recently~\cite{BEOVW}. In fact, once one such system was found, other nonisomorphic
systems with the same parameters were found.

Similarly, to Steiner systems, simple necessary divisibility conditions
for the existence of a given $q$-Steiner system were developed~\cite{ScEt02,Suz90}.

\begin{theorem}
\label{thm:derived}
If a $q$-Steiner system $\dS_q (t,k,n)$ exists, then
for each $i$, $1 \leq i \leq t-1$, a $q$-Steiner system
$\dS_q(t-i,k-i,n-i)$ exists.
\end{theorem}

\begin{cor}
\label{cor:ness}
If a $q$-Steiner system $\dS_q(t,k,n)$ exists, then for all $0 \leq i \leq t-1$,
$$
\frac{\sbinomq{n-i}{t-i}}{\sbinomq{k-i}{t-i}}
$$
must be integers.
\end{cor}

Deriving new designs from designs in general and $q$-Steiner systems in particular
is an important direction to find new designs and to exclude the possible
existence of other designs. Using $q$-analog of the derived design and the residual designs
it was proved that sometimes the necessary conditions for the existence
of a $q$-Steiner system $\dS_q(t,k,n)$ are not sufficient~\cite{KiLa15}.
The first set of parameters ($t$, $k$, and $n$) for which the existence question of $q$-Steiner systems
is not settled is the parameters for the $q$-analog
of the Fano plane, i.e. the $q$-Steiner systems $\dS_q(2,3,7)$, which
will be called also in this paper the $q$-\emph{Fano plane}. There was a lot of effort to
find whether the $q$-Fano plane, especially for $q=2$, exists or does not exist, e.g.~\cite{BKN15,DePh17,EtVa11a,Tho96}.
All these attempts did not provide any answer to the existence question. It was proved recently
in~\cite{BKN15} that if such system exists for $q=2$, then its automorphism group has a small order.
In~\cite{Etz15} a different approach to consider $q$-Steiner systems was given. This approach is based
on puncturing a possible existing $q$-Steiner systems and considering the parameters of the structure
derived from the punctured systems. Properties of the $q$-Fano plane based on this approach were also discussed.
This approach led to the results in the current paper.

In this paper we present a construction for a design with the same parameters
as the design derived from a $q$-Fano
plane, the residual $q$-Fano plane. The constructed design will be also called
the residual $q$-Fano plane. The construction has many places in which there is flexibility for
many choices which lead to a construction of many such designs.
Our definition for the residual $q$-Steiner system
and the derived $q$-Steiner system result in two structures whose union has the same
size as the related $q$-Steiner system, which is not the case for the definition given in~\cite{KiLa15}
and other possible definitions. This makes the residual $q$-Fano plane obtained by our construction
to be a design which is almost as close as possible to a $q$-Fano plane. This definition
of residual $q$-Steiner system and the construction of the residual $q$-Fano plane
is a new direction for a research to solve the existence question of $q$-Steiner systems in
general and $q$-Fano planes in particular.

The rest of this paper is organized as follows. In Section~\ref{sec:residual}
we present a definition for a residual $q$-Steiner system, explain why this
definition represents the appropriate $q$-analog definition, and compare it to the
other definitions. In Section~\ref{sec:comb_struct} a few combinatorial structures
which are used in the construction are defined and some of their properties are discussed.
In Section~\ref{sec:represent} we will discuss representation of subspaces for our construction.
In Section~\ref{sec:extend} it will be explained how to
extend and expand the subspaces in $\F_q^4$ to subspaces in $\F_q^6$. The
construction of the residual $q$-Fano plane is presented in Section~\ref{sec:construction}, where its correctness is also proved.
Conclusions and future research are discussed in Section~\ref{sec:conclude}. In particular we indicate
on the points in the construction in which there is flexibility to construct many different residual
$q$-Fano planes.

\section{Derived and Residual Designs}
\label{sec:residual}

For a design $S$ on a set $Q$, and an element $x \in Q$, the derived design is defined by
$$
\{ B \setminus \{ x \} ~:~ B \in S,~ x \in B \}~,
$$
and the residual design is defined by
$$
\{ B ~:~ B \in S,~ x \notin B \}~.
$$

In~\cite{KiLa15} there is a simple definition for a $q$-analog of the derived design
and the residual design. For this definition we choose an element $u \in \F_q^n$
and an $(n-1)$-subspace $V \subset \F_q^n$ such that $\Span{ \{ u \} \cup V } = \F_q^n$,
where $\Span{X}$ denote the linear span of $X$.
The derived design of a design $\dS$ over $\F_q$, was defined as

\begin{equation}
\label{eq:derKiLa}
\{ B \cap V ~:~ B \in \dS, ~ u \in B \}~,
\end{equation}
and the residual design of $\dS$, was defined as

\begin{equation}
\label{eq:resKiLa}
\{ B ~:~ B \in \dS,~ B \subset V \}~.
\end{equation}

By these definitions, the derived design and residual design of a $q$-Steiner system are both
designs over $\F_q$. This is on the positive side. On the negative side, the size of their
union is significantly smaller than the size of the design from which they were derived.

We present now a different definition for the $q$-analog of a derived design and a residual design which solves this
problem in the definition of~\cite{KiLa15}. Let $u$ be the unit vector with the unique \emph{one} in the
last coordinate, and $V \deff \{ (x,0) ~:~ x \in \F_q^{n-1} \}$. Also, for a subspace $B \subset \F_q^n$,
let $\cZ(B)$ be the subspace obtained from $B$, by removing the last coordinate
of all the vectors in $B$. The derived and residual designs are defined by

\begin{equation}
\label{eq:derived}
\text{der} (\dS) \deff \{ \cZ (B \cap V) ~:~ B \in \dS, ~ u \in B \}~.
\end{equation}

\begin{equation}
\label{eq:residual}
\text{res} (\dS) \deff \{ \cZ(B) ~:~ B \in \dS, ~ u \notin B \}~.
\end{equation}

The two definitions of the derived design are equivalent, but there is a significant
difference in the two definitions of the residual design. For the new definitions given in~(\ref{eq:derived})
and~(\ref{eq:residual}), we have that $|\dS| = |\text{der} (\dS)| + |\text{res} (\dS)|$, a property that
does not hold for the definitions given in~(\ref{eq:derKiLa}) and~(\ref{eq:resKiLa}). The fact that the union
of the two derived designs has size as the original design is one argument that these definitions
serve better as the $q$-analog of the derived design and the residual design. We continue to examine more properties, but the
examination will relate only to Steiner systems $S(t,k,n)$ or only Steiner triple system $S(2,3,n)$, which
are the topic of this paper (but, these properties are also true for other parameters). Another argument is that the uncovered pairs
in a residual Steiner triple system $S(2,3,n)$ form a perfect matching (known also as a 1-factor or $S(1,2,n-1)$)
(see the work of Spencer~\cite{Spe68} for the uncovered pairs of triple systems).
The $q$-analog is the uncovered 2-subspaces in a residual design of a $q$-Steimer system
$\dS_q (2,3,n)$. These uncovered pairs form a $q$-Steiner system $\dS_q(1,2,n-1)$ (known also as a 1-spread). Indeed,
the uncovered pairs in the residual $q$-Steiner system defined in (\ref{eq:residual}) are exactly
the $q$-analog of the uncovered pairs of the residual Steiner system.
This property does not exist in the definition given in (\ref{eq:resKiLa}).
A third argument is a consequence of the next theorem.

The union of the derived $q$-Steiner system and the residual $q$-Steiner system
was called in~\cite{Etz15}, the punctured (or 1-punctured)
$q$-Steiner system. But, no such system was constructed in~\cite{Etz15}.
In the exposition given in~\cite{Etz15} it was proved that

\begin{theorem}
\label{thm:DerRes}
If $\dS$ is a $q$-Steiner system $\dS_q(t,k,n)$, then the derived system
contains exactly $\frac{\sbinomq{n-1}{t-1}}{\sbinomq{k-1}{t-1}}$ distinct $(k-1)$-subspaces
which form a $q$-Steiner system $\dS_q(t-1,k-1,n-1)$. Each $t$-subspace of $\F_q^{n-1}$
which is contained in a $(k-1)$-subspace of $\text{der} (\dS)$ is not contained
in any of the $k$-subspaces of $\text{res}(\dS)$.
Each $t$-subspace of $\F_q^{n-1}$ which is not contained in a $(k-1)$-subspace of $\text{der}(\dS)$,
appears exactly $q^t$ times in the $k$-subspaces of $\text{res}(\dS)$.
\end{theorem}

We will now define any two sets of subspaces which satisfy the properties
given in Theorem~\ref{thm:DerRes} as the derived design and the residual design
for a $q$-Steiner system $\dS_q(t,k,n)$ (but do not depend on the existence
of a $q$-Steiner system $\dS_q(t,k,n)$). For a $q$-Steiner system $\dS_q (t,k,n)$
these definitions are given as follows:

\begin{itemize}
\item A \emph{derived $q$-Steiner system} for a $q$-Steiner system $\dS_q (t,k,n)$ is
a $q$-Steiner system $\dS_q (t-1,k-1,n-1)$.

\item Let $\text{der} (\dS)$ be a $q$-Steiner system $\dS_q(t-1,k-1,n-1)$. The \emph{residual $q$-Steiner system}, $\text{res}(\dS)$,
for a $q$-Steiner system $\dS_q (t,k,n)$ (which might not exists), $\dS$, is a set of distinct $k$-subspaces from $\F_q^{n-1}$
such that each $t$-subspace of $\F_q^{n-1}$ which is not contained in $\text{der} (\dS)$,
is contained in exactly $q^t$ $k$-subspaces of $\text{res}(\dS)$.
\end{itemize}

It should be noted that when $q \to 1$, i.e. for a Steiner system based on an $n$-set, each\linebreak $t$-subset
of the $(n-1)$-set which is not contained in the derived design, is contained in exactly one
$k$-subset of the derived design. This is another indication that our definition for the $q$-analog of the
residual design reflects the best transformation from subsets to subspaces.

It is interesting to know if there exists a system with the same properties of the residual design
in which each $t$-subspace which is not contained in the derived design, is contained in exactly $\lambda$
subspaces of the residual design, where $\lambda < q^t$.
It is not difficult to prove that this is not possible if $\lambda$ is not divisible by $q$
(the proof is left for the interested reader), but it is intriguing to know if $\lambda$ divisible by $q$ is possible.

\section{Combinatorial Structures for the Construction}
\label{sec:comb_struct}

The construction of the residual $q$-Fano plane given in the Section~\ref{sec:construction}
will make use of a few combinatorial structures which are defined, described, and discussed in this section.

The first object is a \emph{1-spread} (spread in short) in $\F_q^n$, where $n$ is
even. A spread $S$ in $\F_q^n$ is a set of 2-subspaces of $\F_q^n$, such that each
nonzero vector of $\F_q^n$ is contained in exactly one 2-subspace of $S$.
It is well known that such a spread exists whenever $n$ is even.

A \emph{1-parallelism} (parallelism in short) in $\F_q^n$ is a partition of the 2-subspaces of $\F_q^n$
into pairwise disjoint spreads. The number of 2-subspaces in such a spread
is $\frac{q^n-1}{q^2-1}$. It was proved by Beutelspacher~\cite{Beu74} that such a parallelism
exists whenever $n$ is a power of 2.

We will be interested in a parallelism in $\F_q^4$, i.e. a partition of the $(q^2+q+1)(q^2+1)$
2-subspaces of $\F_q^4$ into $q^2 +q+1$ disjoint spreads.

We further partition, for our construction of a residual $q$-Fano plane,
the $q^2+q+1$ pairwise disjoint spreads of any given parallelism into three sets $\cA$, $\cB$, and $\cC$.
The set $\cA$ contains one spread. The set $\cB$ contains $q$ spreads, and the set $\cC$ contains
$q^2$ spreads. Any partition of the $q^2+q+1$ spreads is appropriate for this purpose. Such a partition
for $\F_2^4$ is given in Table~\ref{tab:partition}.

In the construction, we have another set $\cD$ which contains all the $q^3 +q^2 +q+1$ distinct
3-subspaces of $\F_q^4$. An example for a basis of the fifteen 3-subspaces of $\F_2^4$ is given
in Table~\ref{tab:3_xubspaces}.

\begin{table}[h]
\centering \caption{Partition of the 2-subspaces of $\F_2^4$ into the sets $\cA$. $\cB$, and $\cC$}
\label{tab:partition}
\begin{small}
\begin{tabular}{|c||c|c|c|c|c||c|c|c|c|c||c|c|c|c|c||c|c|c|c|c|}
\hline
% after \\: \hline or \cline{col1-col2} \cline{col3-col4} ...
$\cA$ & \hspace{-0.4cm} $\begin{array}{c}
000 \\
011 \\
011 \\
101
\end{array}$ \hspace{-0.4cm} & \hspace{-0.4cm} $\begin{array}{c}
011 \\
011 \\
101 \\
000
\end{array}$ \hspace{-0.4cm} & \hspace{-0.4cm} $\begin{array}{c}
011 \\
101 \\
011 \\
011
\end{array}$ \hspace{-0.4cm} & \hspace{-0.4cm} $\begin{array}{c}
011 \\
101 \\
000 \\
101
\end{array}$ \hspace{-0.4cm} & \hspace{-0.4cm} $\begin{array}{c}
011 \\
000 \\
101 \\
110
\end{array}$ \hspace{-0.4cm} & & & & & &&&&& &&&&&\\
\hline
$ \cB$ & \hspace{-0.4cm} $\begin{array}{c}
000 \\
011 \\
101 \\
000
\end{array}$ \hspace{-0.4cm} & \hspace{-0.4cm} $\begin{array}{c}
011 \\
000 \\
101 \\
101
\end{array}$ \hspace{-0.4cm} & \hspace{-0.4cm} $\begin{array}{c}
011 \\
101 \\
011 \\
101
\end{array}$ \hspace{-0.4cm} & \hspace{-0.4cm} $\begin{array}{c}
011 \\
101 \\
101 \\
110
\end{array}$ \hspace{-0.4cm} & \hspace{-0.4cm} $\begin{array}{c}
011 \\
011 \\
000 \\
101
\end{array}$ \hspace{-0.4cm} &
% second spread
\hspace{-0.4cm} $\begin{array}{c}
011 \\
101 \\
000 \\
000
\end{array}$ \hspace{-0.4cm} & \hspace{-0.4cm} $\begin{array}{c}
000 \\
011 \\
101 \\
110
\end{array}$ \hspace{-0.4cm} & \hspace{-0.4cm} $\begin{array}{c}
011 \\
101 \\
110 \\
101
\end{array}$ \hspace{-0.4cm} & \hspace{-0.4cm} $\begin{array}{c}
011 \\
011 \\
011 \\
101
\end{array}$ \hspace{-0.4cm} & \hspace{-0.4cm} $\begin{array}{c}
011 \\
000 \\
101 \\
011
\end{array}$ \hspace{-0.4cm}
&&&&& &&&&&\\
\hline
$ \cC$ & \hspace{-0.4cm} $\begin{array}{c}
011 \\
101 \\
110 \\
000
\end{array}$ \hspace{-0.4cm} & \hspace{-0.4cm} $\begin{array}{c}
011 \\
101 \\
011 \\
110
\end{array}$ \hspace{-0.4cm} & \hspace{-0.4cm} $\begin{array}{c}
011 \\
011 \\
101 \\
011
\end{array}$ \hspace{-0.4cm} & \hspace{-0.4cm} $\begin{array}{c}
011 \\
000 \\
000 \\
101
\end{array}$ \hspace{-0.4cm} & \hspace{-0.4cm} $\begin{array}{c}
000 \\
011 \\
101 \\
101
\end{array}$ \hspace{-0.4cm} &
% second spread
\hspace{-0.4cm} $\begin{array}{c}
011 \\
101 \\
101 \\
000
\end{array}$ \hspace{-0.4cm} & \hspace{-0.4cm} $\begin{array}{c}
011 \\
011 \\
101 \\
101
\end{array}$ \hspace{-0.4cm} & \hspace{-0.4cm} $\begin{array}{c}
011 \\
000 \\
011 \\
101
\end{array}$ \hspace{-0.4cm} & \hspace{-0.4cm} $\begin{array}{c}
000 \\
011 \\
101 \\
011
\end{array}$ \hspace{-0.4cm} & \hspace{-0.4cm} $\begin{array}{c}
011 \\
101 \\
000 \\
011
\end{array}$ \hspace{-0.4cm} &
% third spread
\hspace{-0.4cm} $\begin{array}{c}
011 \\
101 \\
110 \\
011
\end{array}$ \hspace{-0.4cm} & \hspace{-0.4cm} $\begin{array}{c}
011 \\
101 \\
000 \\
110
\end{array}$ \hspace{-0.4cm} & \hspace{-0.4cm} $\begin{array}{c}
011 \\
101 \\
011 \\
000
\end{array}$ \hspace{-0.4cm} & \hspace{-0.4cm} $\begin{array}{c}
011 \\
101 \\
101 \\
101
\end{array}$ \hspace{-0.4cm} & \hspace{-0.4cm} $\begin{array}{c}
000 \\
000 \\
011 \\
101
\end{array}$ \hspace{-0.4cm} &
% fourth spread
\hspace{-0.4cm} $\begin{array}{c}
011 \\
011 \\
101 \\
110
\end{array}$ \hspace{-0.4cm} & \hspace{-0.4cm} $\begin{array}{c}
011 \\
101 \\
101 \\
011
\end{array}$ \hspace{-0.4cm} & \hspace{-0.4cm} $\begin{array}{c}
000 \\
011 \\
000 \\
101
\end{array}$ \hspace{-0.4cm} & \hspace{-0.4cm} $\begin{array}{c}
011 \\
000 \\
101 \\
000
\end{array}$ \hspace{-0.4cm} & \hspace{-0.4cm} $\begin{array}{c}
011 \\
101 \\
110 \\
110
\end{array}$ \hspace{-0.4cm} \\
\hline
\end{tabular}
\end{small}
\end{table}

\begin{table}[h]
\centering \caption{A basis for each one of the fifteen 3-subspaces of the set $\cD$ for $\F_2^4$}
\label{tab:3_xubspaces}
\begin{small}
\begin{tabular}{|c|c|c|c|c|c|c|c|c|c|c|c|c|c|c|}
  \hline
   $\cY_1$&$\cY_2$&$\cY_3$&$\cY_4$&$\cY_5$&$\cY_6$&$\cY_7$&$\cY_8$&$\cY_9$&$\cY_{10}$&$\cY_{11}$&$\cY_{12}$&$\cY_{13}$&$\cY_{14}$&$\cY_{15}$ \\ \hline
\hspace{-0.4cm} $\begin{array}{c}
000 \\
010 \\
011 \\
100
\end{array}$ \hspace{-0.4cm} &
\hspace{-0.4cm} $\begin{array}{c}
001 \\
010 \\
010 \\
100
\end{array}$ \hspace{-0.4cm} &
\hspace{-0.4cm} $\begin{array}{c}
001 \\
010 \\
011 \\
100
\end{array}$ \hspace{-0.4cm} &
\hspace{-0.4cm} $\begin{array}{c}
010 \\
011 \\
100 \\
000
\end{array}$ \hspace{-0.4cm} &
\hspace{-0.4cm} $\begin{array}{c}
010 \\
010 \\
100 \\
001
\end{array}$ \hspace{-0.4cm} &
\hspace{-0.4cm} $\begin{array}{c}
010 \\
011 \\
100 \\
001
\end{array}$ \hspace{-0.4cm} &
\hspace{-0.4cm} $\begin{array}{c}
010 \\
100 \\
010 \\
011
\end{array}$ \hspace{-0.4cm} &
\hspace{-0.4cm} $\begin{array}{c}
010 \\
100 \\
011 \\
011
\end{array}$ \hspace{-0.4cm} &
\hspace{-0.4cm} $\begin{array}{c}
010 \\
100 \\
011 \\
010
\end{array}$ \hspace{-0.4cm} &
\hspace{-0.4cm} $\begin{array}{c}
010 \\
100 \\
000 \\
101
\end{array}$ \hspace{-0.4cm} &
\hspace{-0.4cm} $\begin{array}{c}
010 \\
100 \\
001 \\
100
\end{array}$ \hspace{-0.4cm} &
\hspace{-0.4cm} $\begin{array}{c}
010 \\
100 \\
001 \\
101
\end{array}$ \hspace{-0.4cm} &
\hspace{-0.4cm} $\begin{array}{c}
010 \\
000 \\
100 \\
111
\end{array}$ \hspace{-0.4cm} &
\hspace{-0.4cm} $\begin{array}{c}
010 \\
001 \\
100 \\
110
\end{array}$ \hspace{-0.4cm} &
\hspace{-0.4cm} $\begin{array}{c}
010 \\
001 \\
100 \\
111
\end{array}$ \hspace{-0.4cm}
\\ \hline
\end{tabular}
\end{small}
\end{table}

Let $\alpha$ be a primitive element in $\F_q$.
The next structure that has to be considered is a set of $q^2$ different matrices of size $2 \times (q+1)$ over $\F_q$.
These matrices must satisfy the following properties:
\begin{enumerate}
\item Let $v_1,v_2,\ldots , v_{q+1}$ be the $q+1$ consecutive columns of such a matrix.
For each $i$, $3 \leq i \leq q+1$, $v_i = \alpha^{i-3} v_1 + v_2$
(a scalar $\beta$ is multiplied by each element of a vector~$v$ in the product $\beta v$,
and the vector addition $v_1 + v_2$ is performed element by element in~$\F_q$.).

\item The set of $q^2$ matrices form a linear subspace of dimension two over $\F_q$.

\item For each $i$, $1 \leq i \leq q+1$, the $q^2$ $i$-th column vectors in the $q^2$ matrices
are all distinct, i.e. they consist of all possible $q^2$ column vectors of length 2.
\end{enumerate}
Since these $q^2$ matrices form a linear subspace, it follows that there union is a linear code.
In the sequel, this code will be called the \emph{extension code}.

\begin{lemma}
For each power of a prime $q$ there exists an extension code.
\end{lemma}
\begin{proof}
We start with two $2 \times (q+1)$ matrices over $\F_q$ which will be the basis of the code.

For the first matrix $M_1$ the first column will be
${\tiny \left( \hspace{-0.1cm} \begin{array}{c} 1 \\ 0 \end{array} \hspace{-0.1cm} \right)}$
and the second column will be
${\tiny \left( \hspace{-0.1cm} \begin{array}{c} 0 \\ 1 \end{array} \hspace{-0.1cm} \right)}$.
The $i$-th column, $3 \leq i \leq q+1$, is $\alpha^{i-3}
{\tiny \left( \hspace{-0.1cm} \begin{array}{c} 1 \\ 0 \end{array} \hspace{-0.1cm} \right)}+
{\tiny \left( \hspace{-0.1cm} \begin{array}{c} 0 \\ 1 \end{array} \hspace{-0.1cm} \right)}=
{\tiny \left( \hspace{-0.1cm} \begin{array}{c} \alpha^{i-3} \\ 1 \end{array} \hspace{-0.1cm} \right)}$.

For the second matrix $M_2$ the first column will be
${\tiny \left( \hspace{-0.1cm} \begin{array}{c} 0 \\ 1 \end{array} \hspace{-0.1cm} \right)}$
and the second column will be
${\tiny \left( \hspace{-0.1cm} \begin{array}{c} 1 \\ \beta \end{array} \hspace{-0.1cm} \right)}$,
where $\beta \in \F_q$. The $i-$th column, $3 \leq i \leq q+1$, is $\alpha^{i-3}
{\tiny \left( \hspace{-0.1cm} \begin{array}{c} 0 \\ 1 \end{array} \hspace{-0.1cm} \right)}+
{\tiny \left( \hspace{-0.1cm} \begin{array}{c} 1 \\ \beta \end{array} \hspace{-0.1cm} \right)}=
{\tiny \left( \hspace{-0.1cm} \begin{array}{c} 1 \\ \alpha^{i-3} + \beta \end{array} \hspace{-0.1cm} \right)}$.
We have to prove that there exists a $\beta \in \F_q$ such that the requirements for
the extension code are satisfied.

For this proof we form a $(q+1) \times (q+1)$ matrix $\cM$ whose first row consists of the columns
of the matrix $M_1$ in their given order. The other rows are indexed by the elements of $\F_q$. The row which are indexed
by $\beta \in \F_q$ has
${\tiny \left( \hspace{-0.1cm} \begin{array}{c} 0 \\ 1 \end{array} \hspace{-0.1cm} \right)}$ in the first entry and
${\tiny \left( \hspace{-0.1cm} \begin{array}{c} 1 \\ \beta \end{array} \hspace{-0.1cm} \right)}$ in the second entry.
The $i$-th entry, $3 \leq i \leq q+1$, will be
${\tiny \left( \hspace{-0.1cm} \begin{array}{c} 1 \\ \alpha^{i-3} + \beta \end{array} \hspace{-0.1cm} \right)}$.
It is easy to verify that the $q \times q$ sub-matrix $\cM'$ of $\cM$ defined by removing the first row
and first column of $\cM$ is a Latin square (each row and each column is a permutation of the
$q$ column vectors ${\tiny \left( \hspace{-0.1cm} \begin{array}{c} 1 \\ \beta \end{array} \hspace{-0.1cm} \right)}$,
$\beta \in \F_q$). For each $i$, $3 \leq i \leq q+1$, the element in the $i$-th entry of the first row
of $\cM$ appears in the linear span of the $i$-th entry of exactly one row of $\cM'$. Since $\cM'$ has $q$ rows, it follows that there
exists at least one row which share no linearly dependent entry with the first row of $\cM$. The $\beta$ of such a row
is the required $\beta$ for $M_2$.

The two matrices $M_1$ and $M_2$ are linearly independent. In fact, for each $i$, $1 \leq i \leq q+1$,
the $i$-th columns of the two matrices are linearly independent. Hence, the linear span of $M_1$ and $M_2$
form a linear subspace of dimension two and for the each $i$, $1 \leq i \leq q+1$,
the $i$-th columns of al matrices in the code are distinct.
\end{proof}

Next, we consider all matrices which are candidates for the extension code.
This set of candidates consists of all the $q^4$ distinct $2 \times (q+1)$ matrices over $\F_q$.
If $v_1,v_2,\ldots , v_{q+1}$ are the $q+1$ consecutive columns of such a matrix,
then for each $i$, $3 \leq i \leq q+1$, $v_i = \alpha^{i-3} v_1 + v_2$.
This set of matrices is clearly a linear subspace which will be called the \emph{extension space}.
Since the entries of the first two column vectors can be chosen arbitrarily, it follows that
there are~$q^4$ matrices in the extension space. Moreover, these~$q^4$ matrices form a linear
subspace of dimension four over~$\F_q$. Since the extension code is a linear subspace of
dimension two of the extension space, it follows that we can partition the~$q^4$ matrices of
the extension space into~$q^2$ sets of size~$q^2$ having the following properties:
\begin{enumerate}
\item The extension code is the first set.

\item Let $v_1,v_2,\ldots , v_{q+1}$ be the $q+1$ consecutive columns of any matrix in any of the codes.
For each $i$, $3 \leq i \leq q+1$, $v_i = \alpha^{i-3} v_1 + v_2$.

\item For each $i$, $1 \leq i \leq q+1$, the $q^2$ $i$-th column vectors in the $q^2$ matrices,
of any of the $q^2$ sets,
are all distinct, i.e. they consist of all possible column vectors of length 2.
\end{enumerate}

An example for an extension space (the extension code and its coset) is given in Table~\ref{tab:extension}.

\begin{table}[h]
\centering \caption{The extension space for $q=2$ ($C$ is the code and $C_i$, $i=1,2,3$, are its cosets)}
\label{tab:extension}
\begin{small}
\begin{tabular}{|c||c|c|c|c||c||c|c|c|c||c||c|c|c|c||c||c|c|c|c|}
\hline
\hspace{-0.2cm} $C$ \hspace{-0.2cm} & \hspace{-0.4cm} $\begin{array}{c}
000 \\
000
\end{array}$ \hspace{-0.4cm} & \hspace{-0.4cm} $\begin{array}{c}
011 \\
110
\end{array}$ \hspace{-0.4cm} & \hspace{-0.4cm} $\begin{array}{c}
101 \\
011
\end{array}$ \hspace{-0.4cm} & \hspace{-0.4cm} $\begin{array}{c}
110 \\
101
\end{array}$ \hspace{-0.4cm} &
% second spread
\hspace{-0.2cm} $C_1$ \hspace{-0.2cm} & \hspace{-0.4cm} $\begin{array}{c}
000 \\
110
\end{array}$ \hspace{-0.4cm} & \hspace{-0.4cm} $\begin{array}{c}
011 \\
000
\end{array}$ \hspace{-0.4cm} & \hspace{-0.4cm} $\begin{array}{c}
101 \\
101
\end{array}$ \hspace{-0.4cm} & \hspace{-0.4cm} $\begin{array}{c}
110 \\
011
\end{array}$ \hspace{-0.4cm} &
% third spread
\hspace{-0.2cm} $C_2$ \hspace{-0.2cm} & \hspace{-0.4cm} $\begin{array}{c}
110 \\
000
\end{array}$ \hspace{-0.4cm} & \hspace{-0.4cm} $\begin{array}{c}
101 \\
110
\end{array}$ \hspace{-0.4cm} & \hspace{-0.4cm} $\begin{array}{c}
011 \\
011
\end{array}$ \hspace{-0.4cm} & \hspace{-0.4cm} $\begin{array}{c}
000 \\
101
\end{array}$ \hspace{-0.4cm} &
% fourth spread
\hspace{-0.2cm} $C_3$ \hspace{-0.2cm} & \hspace{-0.4cm} $\begin{array}{c}
110 \\
110
\end{array}$ \hspace{-0.4cm} & \hspace{-0.4cm} $\begin{array}{c}
101 \\
000
\end{array}$ \hspace{-0.4cm} & \hspace{-0.4cm} $\begin{array}{c}
011 \\
101
\end{array}$ \hspace{-0.4cm} & \hspace{-0.4cm} $\begin{array}{c}
000 \\
011
\end{array}$ \hspace{-0.4cm} \\
\hline
\end{tabular}
\end{small}
\end{table}

The construction of the residual $q$-Fano plane will start from sets of subspaces
from $\F_q^4$. The subspaces of these sets will be extended in various ways to 3-subspaces
of $\F_q^6$, in a way that all these extensions will result in the the residual $q$-Fano plane.
The extension space will have an important role in these extensions as will be explained
in Sections~\ref{sec:represent} and~\ref{sec:construction}.
The methods in which subspaces are extended is explained in Section~\ref{sec:represent}.

We end this section with a connection between the subspaces of $\cA$ and the subspaces
of the set $\cD$.

\begin{lemma}
\label{lem:q+1}
A 2-subspace $X$ of $\F_q^4$ can be expanded in $q+1$ distinct ways to a 3-subspace of~$\F_q^4$.
\end{lemma}
\begin{proof}
A 2-subspace $X$ has $q+1$ pairwise linearly independent vectors. $\F_q^4$ has $\frac{q^4-1}{q-1}=q^3+q^2+q+1$
pairwise linearly independent vectors. Each one of the $q^3+q^2$ pairwise linearly independent vectors not in $X$
can be used to for a 3-subspace of $\F_q^4$. Each 3-subspace contain $\frac{q^3-1}{q-1}=q^2+q+1$ pairwise
linearly independent vectors, i.e. $q^2$ additional vector to $X$. Each one of them will form the same
3-subspace when appended to $X$. Hence $X$ can be expanded in $\frac{q^3 + q^2}{q^2} = q+1$ distinct ways
to a 3-subspace of~$\F_q^4$.
\end{proof}
\begin{lemma}
\label{lem:representDbyA}
Each 3-subspace of $\F_q^4$ (also of $\cD$) contains a unique 2-subspace of the set $\cA$.
\end{lemma}
\begin{proof}
If $X \in \cA$ and $v \in \F_q^4$ is a vector such that $v \notin \cA$,
then $Y \deff \Span{ X \cup \{ v \} }$ is clearly a 3-subspace of $\F_q^4$.
Since $Y$ is a 3-subspace and all the 2-subspaces of $\cA$ are pairwise disjoint,
it follows that $Y$ cannot contain two 2-subspaces of $\cA$.

There are $q+1$ different 3-subspaces which contain $X$, $q^2+1$ different 2-subspaces
in $\cA$, and hence there are $(q^2+1)(q+1) = \frac{q^4-1}{q-1}$ 3-subspaces which
contain 2-subspaces from $\cA$. The total number of different 3-subspace
of $\F_q^4$ is $\frac{q^4-1}{q-1}$. It implies that
each 3-subspace of $\F_q^6$ contains a unique 2-subspace of the set $\cA$.
\end{proof}

\section{Representation of Subspaces}
\label{sec:represent}

The construction of the derived $q$-Fano plane and the
residual $q$-Fano plane will be presented in Section~\ref{sec:construction}.
The construction will start with subspaces from $\F_q^4$ which will consists of the
unique 0-subspace of $\F_q^4$ and the
subspaces of the sets $\cA$, $\cB$, $\cC$, and $\cD$. These subspaces will
be extended and/or expanded to 2-subspaces in $\F_q^6$ for the derived $q$-Fano plane, and
to 3-subspaces in $\F_q^6$ for the residual $q$-Fano plane. Most of these extensions will be performed with
the extension space and hence the representations of these subspaces and the matrices of the extension space
must be matched in their representation to make sure that the outcome will be subspaces with the required properties.
To make these extensions and/or expansions simple to explain we will use certain representations of
2-subspaces and 3-subspaces of $\F_q^4$, and 2-subspaces and 3-subspaces of $\F_q^6$.
These representations will also help to verify the correctness of the construction.
For these representations we form an order between the vectors of length 4 of $\F_q^4$.
For simplicity we will use the standard lexicographic order from the smallest to the largest element.

In the representations which follows we will take only one of the $q-1$ different vectors
from which any two are linearly dependent, i.e., $q+1$ vectors for a 2-subspace and $q^2+q+1$
vectors for a 3-subspace. W.l.o.g. (without loss of generality) the vectors which will be taken will always be those whose
first nonzero element is a \emph{one}.

\vspace{0.2cm}

\noindent
{\bf Representation of 2-Subspaces of $\F_q^r$, $r \in \{ 4,6 \}$:}

\vspace{0.2cm}

A 2-subspace X of $\F_q^r$ will be presented by an $r \times (q+1)$ matrix $M$
and an \emph{expanded representation} by an $r \times (q^2+q+1)$ matrix $E(M)$ (or $E(X)$) as follows.
The first $q+1$ columns of the matrices ($M$ and $E(M)$) will be the $q+1$
vectors of length $r$ of $X$, where each two columns are linearly independent
(let us denote these $q+1$ columns by $Y$), with the following two properties:
\begin{itemize}
\item Any two columns of the $4 \times (q+1)$ matrix defined by the first 4 rows and the first $q+1$ columns
of $Y$ are linearly independent, and hence form a basis for $Y$.

\item Let $v_1 v_2 \ldots v_{q+1}$ be the consecutive columns of the matrix defined by
the first four rows and the first $q+1$ columns of $Y$. The first two columns are the smallest
among the $v+1$ columns in the given lexicographic order and $v_1 < v_2$.
Furthermore, $v_i = \alpha^{i-3} v_1 + v_2$, $3 \leq i \leq q+1$.
\end{itemize}
This completes the definition of $M$. For the definition of $E(M)$,
the next column (the\linebreak $(q+2)$-th column) will be an all-zero column. The next
(and last) $(q-1)(q+1)$ columns will consists of $q-1$ identical copies of $Y$.

Any 2-subspace which cannot be represented in this way will not be considered for this
representation (These are 2-subspaces of $\F_q^6$ which have vectors starting with four \emph{zeroes}.).

The 2-subspaces in Table~\ref{tab:partition} are represented by this definition.

\vspace{0.2cm}

\noindent
{\bf Representation of 3-Subspaces of $\F_q^r$, $r \in \{ 4,6 \}$:}

\vspace{0.2cm}

A 3-subspace X of $\F_q^r$ will be presented by an $r \times (q^2+q+1)$ matrix $M$ as follows.
The first $q+1$ columns of $M$ will be the $q+1$
vectors of length $r$ of a 2-subspace of $X$, where each two columns are linearly independent
(let us denote these $q+1$ columns by $Y$), with the following two properties:
\begin{itemize}
\item The first $q+1$ columns of the $4 \times (q+1)$ matrix defined by the first 4 rows and the first $q+1$ columns
of $Y$ represent a 2-subspace of $\cA$, whose existence is guaranteed by Lemma~\ref{lem:representDbyA}.

\item Let $v_1 v_2 \ldots v_{q+1}$ be the consecutive columns of the matrix defined by
the first four rows and the first $q+1$ columns of $Y$. The first two columns are the smallest
among the $v+1$ columns in the given order and $v_1 < v_2$.
Furthermore, $v_i = \alpha^{i-3} v_1 + v_2$, $3 \leq i \leq q+1$.
\end{itemize}
The next column of $M$ (the $(q+2)$-th column) will be a non-zero column vector $v$ of
length~$r$ linearly independent of the first $q+1$ columns of $M$ (or $Y$).
It will be taken as the smallest vector, in the lexicographic order, among the other columns of $X$.
The next $(q-1)(q+1)$ columns of $M$ will consists
of $q-1$ $r \times (q+1)$ matrices, where the $i$-th matrix, $0 \leq i \leq q-2$, is $\alpha^i v +Y$
(the addition of a column vector $v$ of length $r$ to an $r \times m$ matrix $Y$ is done by
adding $v$ to each column of $Y$.). Hence, any two of the first $q+1$ columns with the $(q+2)$-th column
form a basis for the 3-subspace.

After describing the representations of 2-subspaces and 3-subspaces, we are in a position to describe
how we extend and expand a subspace in $\F_q^4$ to a subspace in $\F_q^6$, while keeping these representations.
To make these extensions and expansions simple, we will give a few properties of our representations whose
proofs are trivial. First let $u_i$ ($u'_i$), $1 \leq i \leq q^2+q+1$, be the $i$-th column in the representation
of two distinct subspaces.

\begin{lemma}
In the representation of a 3-subspace $u_1$, $u_2$, $u_{q+2}$ are linearly independent.
\end{lemma}
\begin{lemma}
If for a given 3-subspace and $1 \leq i < j < k \leq q^2+q+1$ we have $\gamma_i u_i +\gamma_j u_j + \gamma_k u_k =0$,
where $\gamma_i , \gamma_j , \gamma_k \in \F_q$,
then for another subspace (of dimension two or three) we have $\gamma_i u'_i +\gamma_j u'_j + \gamma_k u'_k =0$.
\end{lemma}
\begin{lemma}
\label{lem:2in3}
Any 2-subspace $X$ of a 3-subspace $Y$ contains either all the $q+1$ first columns of $Y$ or exactly one
of the first $q+1$ columns of $Y$.
\end{lemma}

\begin{lemma}
\label{lem:exact2in3}
There exists a set $\cP$ which contains $q^2+q+1$ subsets of $\{1,2, \ldots , q^2+q+1 \}$,
each subset is of size $q+1$, such that the columns of the $q^2+q+1$ 2-subspaces of
any $r \times (q^2+q+1)$ matrix $M$, $r \in \{ 4,6 \}$, which represents a 3-subspace, are exactly
on the coordinates of the subsets of $\cP$.
\end{lemma}

\section{Extensions and Expansions of Subspaces}
\label{sec:extend}

The construction of the derived $q$-Fano plane
and the residual $q$-Fano plane will start with 2-subspaces
and 3-subspaces of $\F_q^4$. They will be extended and possibly expanded
to 3-subspaces of $\F_q^6$. We start with a formal definition of the expansion,
which was mentioned before in the representation $E(X)$ of a 2-subspace $X$.

The expansion $E(M,u)$ of an $r \times (q+1)$ matrix $M$, having columns $v_1,v_2,\ldots,v_{q+1}$,
with a column vector $u$ of length $r$ to an $r \times (q^2+q+1)$ matrix as follows.
The next column $v_{q+2}$ is~$u$, and the next $q^2-1$ columns consists
of $q-1$ $r \times (q+1)$ matrices, where the $i$-th matrix is $\alpha^i u +M$.
We note that if $M$ represent a 2-subspace $X$ and $u$ is linearly independent in
the columns of $X$ (i.e. $M$) then $E(M,u)$ represent a 3-subspace. If $M$ represents
a 2-subspaces we can write $E(X,u)$ instead of $E(M,u)$.

The following simple lemmas which were also proved in~\cite{Etz15} provide some of the
foundations for the extensions (with possible expansions).

\begin{lemma}
\label{lem:2To2}
Each 2-subspace in $\F_q^r$ has exactly $q^2$ distinct extensions to a 2-subspace in $\F_q^{r+1}$.
\end{lemma}
\begin{lemma}
\label{lem:2To3}
Each 2-subspace in $\F_q^r$ has a unique extension (with expansion) to a 3-subspace in $\F_q^{r+1}$.
\end{lemma}
\begin{lemma}
\label{lem:3To3}
Each 3-subspace in $\F_q^r$ has exactly $q^3$ distinct extensions to a 3-subspace in~$\F_q^{r+1}$.
\end{lemma}
\begin{lemma}
\label{lem:extend}
Each 2-subspace in $\F_q^4$ has exactly $q^4$ distinct extensions to a 2-subspace in $\F_q^6$.
Each such extension is done by a different $2 \times (q+1)$ matrix of the extension space.
\end{lemma}

In the extensions with possible expansions required in our construction, these lemmas are implemented as follows.

\noindent
{\bf Extension of a 2-subspace from $\F_q^4$ to a 2-subspace of $\F_q^6$:}

\vspace{0.2cm}

Let $X$ be any 2-subspace of $\F_q^4$ which is going to be extended to a 2-subspace
of~$\F_q^6$.
This extension can be done in two steps:
\begin{enumerate}
\item Choose a $2 \times (q+1)$ matrix $Z$ from the extension space.

\item Form a $6 \times (q+1)$ representation matrix for a 2-subspace whose
first four rows is the $4 \times (q+1)$ matrix representation of $X$ and last two rows is $Z$.
\end{enumerate}

\begin{lemma}
\label{lem:q4extensions}
A 2-subspace $X$ of $\F_q^4$ can be extended in $q^4$ distinct ways to a 2-subspace of~$\F_q^6$.
\end{lemma}
\begin{proof}
There are $q^4$ distinct ways to choose an extension matrix $Z$ from the extension space.
Each one yields a different 2-subspace of $\F_q^6$ and all extensions can be formed in this way.
\end{proof}

\vspace{0.2cm}

\noindent
{\bf Extension of a 2-subspace from $\F_q^4$ to a 3-subspace of $\F_q^6$:}

\vspace{0.2cm}

There are two distinct ways to extend a 2-subspace $X$ of $\F_q^4$ to a 3-subspace
of $\F_q^6$.

One way is to extend $X$ first to one of the $q^2$ distinct 2-subspaces of $\F_q^5$ and then
use a unique extension (with expansion) to a 3-subspace of $\F_q^6$.
This is done by extending $X$ to a\linebreak 2-subspace $\tilde{X}$ of $\F_q^6$ by appending to $X$ any one of
the $q^2$ matrices of the extension space whose second row is an all-zero row.
The unique 3-subspace of $\F_q^6$ is obtain by expanding $X$ with $e_6$, the unit vector
of length 6 with the \emph{one} in the last position. Hence, the final 3-subspace
is $E(\tilde{X},e_6)$. Therefore, there are $q^2$ distinct ways for this extension (with expansion).

The second way is to extend $X$ in a unique way (with expansion) to a 3-subspace of $\F_q^5$.
The 3-subspace can be extended in $q^3$ distinct ways to a 3-subspace of $\F_q^6$.
This is done first by appending to $X$ an all-zero row and expand is with $e_5$, the unit vector
of length 5 with the \emph{one} in the last position. There are $q^3$ ways to extend the 3-subspace
of $\F_q^5$ to a 3-subspace of $\F_q^6$. This is done either by using any of the $q^3$ linear
combinations of the first five rows to form the 6-th row, or by taking any of the $q^3$ assignments
from $\F_q$ to positions 1, 2, and $q+2$, and the other positions are fixed by the linear combinations
of the other columns.

\begin{lemma}
\label{lem:q4extensions}
A 2-subspace $X$ of $\F_q^4$ can be extended in $q^3+q^2$ distinct ways to a 3-subspace of~$\F_q^6$.
\end{lemma}

\vspace{0.2cm}

\noindent
{\bf Extension of a 3-subspace from $\F_q^4$ to a 3-subspace of $\F_q^6$:}

Let $Y$ be any 3-subspace of $\F_q^4$ which is going to be extended to a 3-subspace
of $\F_q^6$.
This extension can be done in four steps:
\begin{enumerate}
\item Choose a $2 \times (q+1)$ matrix $Z$ from the extension space.

\item Choose a column vector $u$ of length two over $\F_q$.

\item Form the $2 \times (q^2 +q+1)$ expansion $E(Z,u)$.

\item Form a $6 \times (q^2+q+1)$ representation matrix for a 3-subspace whose
first four rows is the matrix representation of $Y$ and last two rows is $E(Z,u)$.
\end{enumerate}

\begin{lemma}
A 3-subspace $Y$ of $\F_q^4$ can be extended in $q^6$ distinct ways to a 3-subspace of $\F_q^6$.
\end{lemma}
\begin{proof}
There are $q^4$ distinct ways to choose an extension matrix $Z$ from the extension space
and $q^2$ way to choose the vector $u$ for $E(Z,u)$. Each such choice will yield a different
3-subspace of $\F_q^6$ since the process starts with a 3-subspace. To complete the proof
we note that each extension can be formed in this way.
\end{proof}

\section{Construction of Residual $q$-Fano Planes}
\label{sec:construction}

The construction of the derived $q$-Fano plane and the residual $q$-Fano plane is based on extensions and
possible expansion of the subspaces in the sets $\cA$, $\cB$,
$\cC$, and $\cD$, which contain 2-subspaces and 3-subspaces of $\F_q^4$ into
3-subspaces in $\F_q^6$. These extensions and/or expansions and the extension of the null-subspace
of $\F_q^4$ will form the residual $q$-Fano plane.

There is one possible way to form a 2-subspace of $\F_q^6$ whose first four rows in the
matrix representation corresponds to the 0-subspace of $\F_q^4$.
The set of size one which contains this 2-subspace will be denoted by $\dS_0$.

\vspace{0.2cm}

\noindent
{\bf Extension of Type A:}

\vspace{0.1cm}

The set $\cA$ of 2-subspaces of $\F_q^4$ contains $q^2+1$ subspaces.
Each one is extended in the $q^2$ possible distinct ways, based on the extension code $C$, to a 2-subspace in $\F_q^6$.
The result is a set with $q^4 + q^2$ distinct 2-subspaces of $\F_q^6$. This set will be
denoted by $\dS_\cA$.

\begin{lemma}
\label{lem:spread_A}
The set $\dS_0 \cup \dS_\cA$ is a spread in $\F_q^6$.
\end{lemma}
\begin{proof}
The set $\cA$ is a spread in $\F_q^4$ by definition.
The extension based on $C$ is a 2-subspace in $\F_q^6$.
A spread in $\F_q^6$ contains $\frac{q^6-1}{q^2-1} =q^4+q^2+1$ disjoint 2-subspaces.
$\dS_0$ has one 2-subspace and $\dS_\cA$ contains $q^4+q^2$ 2-subspaces. Hence, to complete
the proof it is sufficient to prove that no nonzero vector of $\F_q^6$ appears more than once in a subspaces
of $\dS_0 \cup \dS_\cA$. Assume a vector $v \in \F_q^6$ appears in two such subspaces. Let $v' \in \F_q^4$
be the prefix vector of length 4 obtained from $v$. By the definition of $\cA$ we have that $v'$ is either
the all-zero vector or it is contained in a unique 2-subspace of $\cA$. If $v'$ is the all-zero vector
then $v$ is contained only in the unique subspace of $\dS_0$. If $v'$ is contained in a unique 2-subspace $X$
of $\cA$, then by the definition of the extension code $C$,
each one of the $q^2$ extensions of $X$ with the extension code $C$ appends a different suffix
of length two to $v'$ and hence $v$ cannot appear more than once.
\end{proof}

Table~\ref{tab:S_A0} presents the 21 2-subspaces of $\dS_0 \cup \dS_\cA$ for $q=2$. The
first four rows in the matrix representation is a 2-subspace of $\dS_\cA$ and the last two rows
are taken from the extension code.

\begin{table}[h]
\centering \caption{the 2-subspaces of $\dS_0 \cup \dS_\cA$ for $q=2$}
\label{tab:S_A0}
\begin{small}
\begin{tabular}{|c||c|c|c|c|c||c|c|c|c|c||c|c|c|c|c||c|c|c|c|c|}
\hline
% after \\: \hline or \cline{col1-col2} \cline{col3-col4} ...
\hspace{-0.4cm} $\begin{array}{c}
000 \\
000 \\
000 \\
000 \\
011 \\
101
\end{array}$ \hspace{-0.4cm} & \hspace{-0.4cm} $\begin{array}{c}
000 \\
011 \\
011 \\
101 \\
000 \\
000
\end{array}$ \hspace{-0.4cm} & \hspace{-0.4cm} $\begin{array}{c}
011 \\
011 \\
101 \\
000 \\
000 \\
000
\end{array}$ \hspace{-0.4cm} & \hspace{-0.4cm} $\begin{array}{c}
011 \\
101 \\
011 \\
011 \\
000 \\
000
\end{array}$ \hspace{-0.4cm} & \hspace{-0.4cm} $\begin{array}{c}
011 \\
101 \\
000 \\
101 \\
000 \\
000
\end{array}$ \hspace{-0.4cm} & \hspace{-0.4cm} $\begin{array}{c}
011 \\
000 \\
101 \\
110 \\
000 \\
000
\end{array}$ \hspace{-0.4cm} &
\hspace{-0.4cm} $\begin{array}{c}
000 \\
011 \\
011 \\
101 \\
011 \\
110
\end{array}$ \hspace{-0.4cm} & \hspace{-0.4cm} $\begin{array}{c}
011 \\
011 \\
101 \\
000 \\
011 \\
110
\end{array}$ \hspace{-0.4cm} & \hspace{-0.4cm} $\begin{array}{c}
011 \\
101 \\
011 \\
011 \\
011 \\
110
\end{array}$ \hspace{-0.4cm} & \hspace{-0.4cm} $\begin{array}{c}
011 \\
101 \\
000 \\
101 \\
011 \\
110
\end{array}$ \hspace{-0.4cm} & \hspace{-0.4cm} $\begin{array}{c}
011 \\
000 \\
101 \\
110 \\
011 \\
110
\end{array}$ \hspace{-0.4cm} &
\hspace{-0.4cm} $\begin{array}{c}
000 \\
011 \\
011 \\
101 \\
101 \\
011
\end{array}$ \hspace{-0.4cm} & \hspace{-0.4cm} $\begin{array}{c}
011 \\
011 \\
101 \\
000 \\
101 \\
011
\end{array}$ \hspace{-0.4cm} & \hspace{-0.4cm} $\begin{array}{c}
011 \\
101 \\
011 \\
011 \\
101 \\
011
\end{array}$ \hspace{-0.4cm} & \hspace{-0.4cm} $\begin{array}{c}
011 \\
101 \\
000 \\
101 \\
101 \\
011
\end{array}$ \hspace{-0.4cm} & \hspace{-0.4cm} $\begin{array}{c}
011 \\
000 \\
101 \\
110 \\
101 \\
011
\end{array}$ \hspace{-0.4cm} &
\hspace{-0.4cm} $\begin{array}{c}
000 \\
011 \\
011 \\
101 \\
110 \\
101
\end{array}$ \hspace{-0.4cm} & \hspace{-0.4cm} $\begin{array}{c}
011 \\
011 \\
101 \\
000 \\
110 \\
101
\end{array}$ \hspace{-0.4cm} & \hspace{-0.4cm} $\begin{array}{c}
011 \\
101 \\
011 \\
011 \\
110 \\
101
\end{array}$ \hspace{-0.4cm} & \hspace{-0.4cm} $\begin{array}{c}
011 \\
101 \\
000 \\
101 \\
110 \\
101
\end{array}$ \hspace{-0.4cm} & \hspace{-0.4cm} $\begin{array}{c}
011 \\
000 \\
101 \\
110 \\
110 \\
101
\end{array}$ \hspace{-0.4cm} \\
\hline
\end{tabular}
\end{small}
\end{table}

\noindent
{\bf Extension of Type B:}

\vspace{0.05cm}

The set $\cB$ of 2-subspaces of $\F_q^4$ contains $q$ spreads with a total of $q(q^2+1)$ subspaces.
Each one is extended in all possible $q^2$ distinct ways to a 2-subspace in $\F_q^5$.
Each such 2-subspace of $\F_q^5$ is extended in a unique way to a 3-subspace in $\F_q^6$.
The result is a set with $q^3(q^2+1)$ distinct 3-subspaces of $\F_q^6$.
This set will be denoted by $\dS_\cB$.
Table~\ref{tab:S_B} presents the forty 3-subspaces of $\dS_\cB$ for $q=2$. Note that the third vector
in all the subspaces is the same.

\begin{table}[h]
\centering \caption{A basis for each one of the forty 3-subspaces of $\dS_\cB$ for $q=2$}
\label{tab:S_B}
\begin{small}
\begin{tabular}{|c|c|c|c||c|c|c|c||c|c|c|c||c|c|c|c||c|c|c|c|}
\hline
% after \\: \hline or \cline{col1-col2} \cline{col3-col4} ...
\hspace{-0.4cm} $\begin{array}{c}
000 \\
010 \\
100 \\
000 \\
000 \\
001
\end{array}$ \hspace{-0.4cm} & \hspace{-0.4cm} $\begin{array}{c}
000 \\
010 \\
100 \\
000 \\
010 \\
001
\end{array}$ \hspace{-0.4cm} & \hspace{-0.4cm} $\begin{array}{c}
000 \\
010 \\
100 \\
000 \\
100 \\
001
\end{array}$ \hspace{-0.4cm} & \hspace{-0.4cm} $\begin{array}{c}
000 \\
010 \\
100 \\
000 \\
110 \\
001
\end{array}$ \hspace{-0.4cm} &
\hspace{-0.4cm} $\begin{array}{c}
010 \\
000 \\
100 \\
100 \\
000 \\
001
\end{array}$ \hspace{-0.4cm} & \hspace{-0.4cm} $\begin{array}{c}
010 \\
000 \\
100 \\
100 \\
010 \\
001
\end{array}$ \hspace{-0.4cm} & \hspace{-0.4cm} $\begin{array}{c}
010 \\
000 \\
100 \\
100 \\
100 \\
001
\end{array}$ \hspace{-0.4cm} & \hspace{-0.4cm} $\begin{array}{c}
010 \\
000 \\
100 \\
100 \\
110 \\
001
\end{array}$ \hspace{-0.4cm} &
\hspace{-0.4cm} $\begin{array}{c}
010 \\
100 \\
010 \\
100 \\
000 \\
001
\end{array}$ \hspace{-0.4cm} & \hspace{-0.4cm} $\begin{array}{c}
010 \\
100 \\
010 \\
100 \\
010 \\
001
\end{array}$ \hspace{-0.4cm} & \hspace{-0.4cm} $\begin{array}{c}
010 \\
100 \\
010 \\
100 \\
100 \\
001
\end{array}$ \hspace{-0.4cm} & \hspace{-0.4cm} $\begin{array}{c}
010 \\
100 \\
010 \\
100 \\
110 \\
001
\end{array}$ \hspace{-0.4cm} & \hspace{-0.4cm} $\begin{array}{c}
010 \\
100 \\
100 \\
110 \\
000 \\
001
\end{array}$ \hspace{-0.4cm} & \hspace{-0.4cm} $\begin{array}{c}
010 \\
100 \\
100 \\
110 \\
010 \\
001
\end{array}$ \hspace{-0.4cm} & \hspace{-0.4cm} $\begin{array}{c}
010 \\
100 \\
100 \\
110 \\
100 \\
001
\end{array}$ \hspace{-0.4cm} & \hspace{-0.4cm} $\begin{array}{c}
010 \\
100 \\
100 \\
110 \\
110 \\
001
\end{array}$ \hspace{-0.4cm} & \hspace{-0.4cm} $\begin{array}{c}
010 \\
010 \\
000 \\
100 \\
000 \\
001
\end{array}$ \hspace{-0.4cm} & \hspace{-0.4cm} $\begin{array}{c}
010 \\
010 \\
000 \\
100 \\
010 \\
001
\end{array}$ \hspace{-0.4cm} & \hspace{-0.4cm} $\begin{array}{c}
010 \\
010 \\
000 \\
100 \\
100 \\
001
\end{array}$ \hspace{-0.4cm} & \hspace{-0.4cm} $\begin{array}{c}
010 \\
010 \\
000 \\
100 \\
110 \\
001
\end{array}$ \hspace{-0.4cm} \\
\hline
% second spread
\hspace{-0.4cm} $\begin{array}{c}
010 \\
100 \\
000 \\
000 \\
000 \\
001
\end{array}$ \hspace{-0.4cm} & \hspace{-0.4cm} $\begin{array}{c}
010 \\
100 \\
000 \\
000 \\
010 \\
001
\end{array}$ \hspace{-0.4cm} & \hspace{-0.4cm} $\begin{array}{c}
010 \\
100 \\
000 \\
000 \\
100 \\
001
\end{array}$ \hspace{-0.4cm} & \hspace{-0.4cm} $\begin{array}{c}
010 \\
100 \\
000 \\
000 \\
110 \\
001
\end{array}$ \hspace{-0.4cm} & \hspace{-0.4cm} $\begin{array}{c}
000 \\
010 \\
100 \\
110 \\
000 \\
001
\end{array}$ \hspace{-0.4cm} & \hspace{-0.4cm} $\begin{array}{c}
000 \\
010 \\
100 \\
110 \\
010 \\
001
\end{array}$ \hspace{-0.4cm} & \hspace{-0.4cm} $\begin{array}{c}
000 \\
010 \\
100 \\
110 \\
100 \\
001
\end{array}$ \hspace{-0.4cm} & \hspace{-0.4cm} $\begin{array}{c}
000 \\
010 \\
100 \\
110 \\
110 \\
001
\end{array}$ \hspace{-0.4cm} & \hspace{-0.4cm} $\begin{array}{c}
010 \\
100 \\
110 \\
100 \\
000 \\
001
\end{array}$ \hspace{-0.4cm} & \hspace{-0.4cm} $\begin{array}{c}
010 \\
100 \\
110 \\
100 \\
010 \\
001
\end{array}$ \hspace{-0.4cm} & \hspace{-0.4cm} $\begin{array}{c}
010 \\
100 \\
110 \\
100 \\
100 \\
001
\end{array}$ \hspace{-0.4cm} & \hspace{-0.4cm} $\begin{array}{c}
010 \\
100 \\
110 \\
100 \\
110 \\
001
\end{array}$ \hspace{-0.4cm} & \hspace{-0.4cm} $\begin{array}{c}
010 \\
010 \\
010 \\
100 \\
000 \\
001
\end{array}$ \hspace{-0.4cm} & \hspace{-0.4cm} $\begin{array}{c}
010 \\
010 \\
010 \\
100 \\
010 \\
001
\end{array}$ \hspace{-0.4cm} & \hspace{-0.4cm} $\begin{array}{c}
010 \\
010 \\
010 \\
100 \\
100 \\
001
\end{array}$ \hspace{-0.4cm} & \hspace{-0.4cm} $\begin{array}{c}
010 \\
010 \\
010 \\
100 \\
110 \\
001
\end{array}$ \hspace{-0.4cm} & \hspace{-0.4cm} $\begin{array}{c}
010 \\
000 \\
101 \\
010 \\
000 \\
001
\end{array}$ \hspace{-0.4cm} & \hspace{-0.4cm} $\begin{array}{c}
010 \\
000 \\
101 \\
010 \\
010 \\
001
\end{array}$ \hspace{-0.4cm} & \hspace{-0.4cm} $\begin{array}{c}
010 \\
000 \\
101 \\
010 \\
100 \\
001
\end{array}$ \hspace{-0.4cm} & \hspace{-0.4cm} $\begin{array}{c}
010 \\
000 \\
101 \\
010 \\
110 \\
001
\end{array}$ \hspace{-0.4cm} \\
\hline
\end{tabular}
\end{small}
\end{table}

\noindent
{\bf Extension of Type C:}

\vspace{0.1cm}

The set $\cC$ of 2-subspaces of $\F_q^4$ contains $q^2$ spreads, each one
has $q^2+1$ subspaces. We further partition $\cC$ into $q$ subsets $\cC_\xi$, $\xi \in \F_q$,
where $\cC_\xi$ contains $q$ spreads.

For each $\xi$, $\xi \in \F_q$, the set $\cC_\xi$ of 2-subspaces of $\F_q^4$ contains
$q(q^2+1)$ subspaces. Each one is extended in a unique way to a 3-subspace in $\F_q^5$.
Each such 3-subspace of $\F_q^5$ has $q^3$ extensions to a 3-subspace in $\F_q^6$. Let $\dS_{\cC_\xi}$
be the set of these 3-subspaces which have $\xi$ in the 6-th row of the
$(q+2)$-th column of the matrix representation. This set $\dS_{\cC_\xi}$
contains $q^3 (q^2+1)$ distinct 3-subspaces of $\F_q^6$ since there are $q^2$ distinct
ways to choose the pair of symbols in the sixth row for the first two linearly
independent vectors of the 3-subspace. If $\dS_\cC \deff \cup_{\xi \in \F_q} \dS_{\cC_\xi}$
then clearly $\dS_\cC$ contains $q^4 (q^2+1)$ distinct 3-subspaces.

Table~\ref{tab:S_C} presents the eighty 3-subspaces of $\dS_\cC$ for $q=2$, where the first
two spreads in Table~\ref{tab:partition} are taken as $\cC_0$ and the other two spreads
form $\cC_1$. Note, that the third vector in the basis of the subspaces from $\dS_{\cC_0}$ and
the one from $\dS_{\cC_1}$ differ exactly in the last entry.

\begin{table}[h]
\centering \caption{A basis for each one of the eighty 3-subspaces of $\dS_\cC$ for $q=2$}
\label{tab:S_C}
\begin{small}
\begin{tabular}{|c||c|c|c|c||c|c|c|c||c|c|c|c||c|c|c|c||c|c|c|c||c|c|c|c|}
% after \\: \hline or \cline{col1-col2} \cline{col3-col4} ...
\hline
$\dS_{\cC_0}$ & \hspace{-0.4cm} $\begin{array}{c}
010 \\
100 \\
110 \\
000 \\
001 \\
000
\end{array}$ \hspace{-0.4cm} & \hspace{-0.4cm} $\begin{array}{c}
010 \\
100 \\
110 \\
000 \\
001 \\
010
\end{array}$ \hspace{-0.4cm} & \hspace{-0.4cm} $\begin{array}{c}
010 \\
100 \\
110 \\
000 \\
001 \\
100
\end{array}$ \hspace{-0.4cm} & \hspace{-0.4cm} $\begin{array}{c}
010 \\
100 \\
110 \\
000 \\
001 \\
110
\end{array}$ \hspace{-0.4cm} & \hspace{-0.4cm} $\begin{array}{c}
010 \\
100 \\
010 \\
110 \\
001 \\
000
\end{array}$ \hspace{-0.4cm} & \hspace{-0.4cm} $\begin{array}{c}
010 \\
100 \\
010 \\
110 \\
001 \\
010
\end{array}$ \hspace{-0.4cm} & \hspace{-0.4cm} $\begin{array}{c}
010 \\
100 \\
010 \\
110 \\
001 \\
100
\end{array}$ \hspace{-0.4cm} & \hspace{-0.4cm} $\begin{array}{c}
010 \\
100 \\
010 \\
110 \\
001 \\
110
\end{array}$ \hspace{-0.4cm} & \hspace{-0.4cm} $\begin{array}{c}
010 \\
010 \\
100 \\
010 \\
001 \\
000
\end{array}$ \hspace{-0.4cm} & \hspace{-0.4cm} $\begin{array}{c}
010 \\
010 \\
100 \\
010 \\
001 \\
010
\end{array}$ \hspace{-0.4cm} & \hspace{-0.4cm} $\begin{array}{c}
010 \\
010 \\
100 \\
010 \\
001 \\
100
\end{array}$ \hspace{-0.4cm} & \hspace{-0.4cm} $\begin{array}{c}
010 \\
010 \\
100 \\
010 \\
001 \\
110
\end{array}$ \hspace{-0.4cm} & \hspace{-0.4cm} $\begin{array}{c}
010 \\
000 \\
000 \\
100 \\
001 \\
000
\end{array}$ \hspace{-0.4cm} & \hspace{-0.4cm} $\begin{array}{c}
010 \\
000 \\
000 \\
100 \\
001 \\
010
\end{array}$ \hspace{-0.4cm} & \hspace{-0.4cm} $\begin{array}{c}
010 \\
000 \\
000 \\
100 \\
001 \\
100
\end{array}$ \hspace{-0.4cm} & \hspace{-0.4cm} $\begin{array}{c}
010 \\
000 \\
000 \\
100 \\
001 \\
110
\end{array}$ \hspace{-0.4cm} & \hspace{-0.4cm} $\begin{array}{c}
000 \\
010 \\
100 \\
100 \\
001 \\
000
\end{array}$ \hspace{-0.4cm} & \hspace{-0.4cm} $\begin{array}{c}
000 \\
010 \\
100 \\
100 \\
001 \\
010
\end{array}$ \hspace{-0.4cm} & \hspace{-0.4cm} $\begin{array}{c}
000 \\
010 \\
100 \\
100 \\
001 \\
100
\end{array}$ \hspace{-0.4cm} & \hspace{-0.4cm} $\begin{array}{c}
000 \\
010 \\
100 \\
100 \\
001 \\
110
\end{array}$ \hspace{-0.4cm} \\
\hline
% second spread
$\dS_{\cC_0}$ &
\hspace{-0.4cm} $\begin{array}{c}
010 \\
100 \\
100 \\
000 \\
001 \\
000
\end{array}$ \hspace{-0.4cm} & \hspace{-0.4cm} $\begin{array}{c}
010 \\
100 \\
100 \\
000 \\
001 \\
010
\end{array}$ \hspace{-0.4cm} & \hspace{-0.4cm} $\begin{array}{c}
010 \\
100 \\
100 \\
000 \\
001 \\
100
\end{array}$ \hspace{-0.4cm} & \hspace{-0.4cm} $\begin{array}{c}
010 \\
100 \\
100 \\
000 \\
001 \\
110
\end{array}$ \hspace{-0.4cm} & \hspace{-0.4cm} $\begin{array}{c}
010 \\
010 \\
100 \\
100 \\
001 \\
000
\end{array}$ \hspace{-0.4cm} & \hspace{-0.4cm} $\begin{array}{c}
010 \\
010 \\
100 \\
100 \\
001 \\
010
\end{array}$ \hspace{-0.4cm} & \hspace{-0.4cm} $\begin{array}{c}
010 \\
010 \\
100 \\
100 \\
001 \\
100
\end{array}$ \hspace{-0.4cm} & \hspace{-0.4cm} $\begin{array}{c}
010 \\
010 \\
100 \\
100 \\
001 \\
110
\end{array}$ \hspace{-0.4cm} & \hspace{-0.4cm} $\begin{array}{c}
010 \\
000 \\
010 \\
100 \\
001 \\
000
\end{array}$ \hspace{-0.4cm} & \hspace{-0.4cm} $\begin{array}{c}
010 \\
000 \\
010 \\
100 \\
001 \\
010
\end{array}$ \hspace{-0.4cm} & \hspace{-0.4cm} $\begin{array}{c}
010 \\
000 \\
010 \\
100 \\
001 \\
100
\end{array}$ \hspace{-0.4cm} & \hspace{-0.4cm} $\begin{array}{c}
010 \\
000 \\
010 \\
100 \\
001 \\
110
\end{array}$ \hspace{-0.4cm} & \hspace{-0.4cm} $\begin{array}{c}
000 \\
010 \\
100 \\
010 \\
001 \\
000
\end{array}$ \hspace{-0.4cm} & \hspace{-0.4cm} $\begin{array}{c}
000 \\
010 \\
100 \\
010 \\
001 \\
010
\end{array}$ \hspace{-0.4cm} & \hspace{-0.4cm} $\begin{array}{c}
000 \\
010 \\
100 \\
010 \\
001 \\
100
\end{array}$ \hspace{-0.4cm} & \hspace{-0.4cm} $\begin{array}{c}
000 \\
010 \\
100 \\
010 \\
001 \\
110
\end{array}$ \hspace{-0.4cm} & \hspace{-0.4cm} $\begin{array}{c}
010 \\
100 \\
000 \\
010 \\
001 \\
000
\end{array}$ \hspace{-0.4cm} & \hspace{-0.4cm} $\begin{array}{c}
010 \\
100 \\
000 \\
010 \\
001 \\
010
\end{array}$ \hspace{-0.4cm} & \hspace{-0.4cm} $\begin{array}{c}
010 \\
100 \\
000 \\
010 \\
001 \\
100
\end{array}$ \hspace{-0.4cm} & \hspace{-0.4cm} $\begin{array}{c}
010 \\
100 \\
000 \\
010 \\
001 \\
110
\end{array}$ \hspace{-0.4cm}  \\
\hline
% third spread
$\dS_{\cC_1}$ & \hspace{-0.4cm} $\begin{array}{c}
010 \\
100 \\
110 \\
010 \\
001 \\
001
\end{array}$ \hspace{-0.4cm} & \hspace{-0.4cm} $\begin{array}{c}
010 \\
100 \\
110 \\
010 \\
001 \\
011
\end{array}$ \hspace{-0.4cm} & \hspace{-0.4cm} $\begin{array}{c}
010 \\
100 \\
110 \\
010 \\
001 \\
101
\end{array}$ \hspace{-0.4cm} & \hspace{-0.4cm} $\begin{array}{c}
010 \\
100 \\
110 \\
010 \\
001 \\
111
\end{array}$ \hspace{-0.4cm} & \hspace{-0.4cm} $\begin{array}{c}
010 \\
100 \\
000 \\
110 \\
001 \\
001
\end{array}$ \hspace{-0.4cm} & \hspace{-0.4cm} $\begin{array}{c}
010 \\
100 \\
000 \\
110 \\
001 \\
011
\end{array}$ \hspace{-0.4cm} & \hspace{-0.4cm} $\begin{array}{c}
010 \\
100 \\
000 \\
110 \\
001 \\
101
\end{array}$ \hspace{-0.4cm} & \hspace{-0.4cm} $\begin{array}{c}
010 \\
100 \\
000 \\
110 \\
001 \\
111
\end{array}$ \hspace{-0.4cm} & \hspace{-0.4cm} $\begin{array}{c}
010 \\
100 \\
010 \\
000 \\
001 \\
001
\end{array}$ \hspace{-0.4cm} & \hspace{-0.4cm} $\begin{array}{c}
010 \\
100 \\
010 \\
000 \\
001 \\
011
\end{array}$ \hspace{-0.4cm} & \hspace{-0.4cm} $\begin{array}{c}
010 \\
100 \\
010 \\
000 \\
001 \\
101
\end{array}$ \hspace{-0.4cm} & \hspace{-0.4cm} $\begin{array}{c}
010 \\
100 \\
010 \\
000 \\
001 \\
111
\end{array}$ \hspace{-0.4cm} & \hspace{-0.4cm} $\begin{array}{c}
010 \\
100 \\
100 \\
100 \\
001 \\
001
\end{array}$ \hspace{-0.4cm} & \hspace{-0.4cm} $\begin{array}{c}
010 \\
100 \\
100 \\
100 \\
001 \\
011
\end{array}$ \hspace{-0.4cm} & \hspace{-0.4cm} $\begin{array}{c}
010 \\
100 \\
100 \\
100 \\
001 \\
101
\end{array}$ \hspace{-0.4cm} & \hspace{-0.4cm} $\begin{array}{c}
010 \\
100 \\
100 \\
100 \\
001 \\
111
\end{array}$ \hspace{-0.4cm} & \hspace{-0.4cm} $\begin{array}{c}
000 \\
000 \\
010 \\
100 \\
001 \\
001 \\
\end{array}$ \hspace{-0.4cm} & \hspace{-0.4cm} $\begin{array}{c}
000 \\
000 \\
010 \\
100 \\
001 \\
011 \\
\end{array}$ \hspace{-0.4cm} & \hspace{-0.4cm} $\begin{array}{c}
000 \\
000 \\
010 \\
100 \\
001 \\
101 \\
\end{array}$ \hspace{-0.4cm} & \hspace{-0.4cm} $\begin{array}{c}
000 \\
000 \\
010 \\
100 \\
001 \\
111 \\
\end{array}$ \hspace{-0.4cm} \\
\hline
% fourth spread
$\dS_{\cC_1}$ & \hspace{-0.4cm} $\begin{array}{c}
010 \\
010 \\
100 \\
110 \\
001 \\
001
\end{array}$ \hspace{-0.4cm} & \hspace{-0.4cm} $\begin{array}{c}
010 \\
010 \\
100 \\
110 \\
001 \\
011
\end{array}$ \hspace{-0.4cm} & \hspace{-0.4cm} $\begin{array}{c}
010 \\
010 \\
100 \\
110 \\
001 \\
101
\end{array}$ \hspace{-0.4cm} & \hspace{-0.4cm} $\begin{array}{c}
010 \\
010 \\
100 \\
110 \\
001 \\
111
\end{array}$ \hspace{-0.4cm} & \hspace{-0.4cm} $\begin{array}{c}
010 \\
100 \\
100 \\
010 \\
001 \\
001
\end{array}$ \hspace{-0.4cm} & \hspace{-0.4cm} $\begin{array}{c}
010 \\
100 \\
100 \\
010 \\
001 \\
011
\end{array}$ \hspace{-0.4cm} & \hspace{-0.4cm} $\begin{array}{c}
010 \\
100 \\
100 \\
010 \\
001 \\
101
\end{array}$ \hspace{-0.4cm} & \hspace{-0.4cm} $\begin{array}{c}
010 \\
100 \\
100 \\
010 \\
001 \\
111
\end{array}$ \hspace{-0.4cm} & \hspace{-0.4cm} $\begin{array}{c}
000 \\
010 \\
000 \\
100 \\
001 \\
001
\end{array}$ \hspace{-0.4cm} & \hspace{-0.4cm} $\begin{array}{c}
000 \\
010 \\
000 \\
100 \\
001 \\
011
\end{array}$ \hspace{-0.4cm} & \hspace{-0.4cm} $\begin{array}{c}
000 \\
010 \\
000 \\
100 \\
001 \\
101
\end{array}$ \hspace{-0.4cm} & \hspace{-0.4cm} $\begin{array}{c}
000 \\
010 \\
000 \\
100 \\
001 \\
111
\end{array}$ \hspace{-0.4cm} & \hspace{-0.4cm} $\begin{array}{c}
010 \\
000 \\
100 \\
000 \\
001 \\
001
\end{array}$ \hspace{-0.4cm} & \hspace{-0.4cm} $\begin{array}{c}
010 \\
000 \\
100 \\
000 \\
001 \\
011
\end{array}$ \hspace{-0.4cm} & \hspace{-0.4cm} $\begin{array}{c}
010 \\
000 \\
100 \\
000 \\
001 \\
101
\end{array}$ \hspace{-0.4cm} & \hspace{-0.4cm} $\begin{array}{c}
010 \\
000 \\
100 \\
000 \\
001 \\
111
\end{array}$ \hspace{-0.4cm} & \hspace{-0.4cm} $\begin{array}{c}
010 \\
100 \\
110 \\
110 \\
001 \\
001
\end{array}$ \hspace{-0.4cm} & \hspace{-0.4cm} $\begin{array}{c}
010 \\
100 \\
110 \\
110 \\
001 \\
011
\end{array}$ \hspace{-0.4cm} & \hspace{-0.4cm} $\begin{array}{c}
010 \\
100 \\
110 \\
110 \\
001 \\
101
\end{array}$ \hspace{-0.4cm} & \hspace{-0.4cm} $\begin{array}{c}
010 \\
100 \\
110 \\
110 \\
001 \\
001
\end{array}$ \hspace{-0.4cm} \\
\hline
\end{tabular}
\end{small}
\end{table}

\noindent
{\bf Extension of Type D:}

\vspace{0.1cm}

First, we partition the cosets of the extension code $C$ (all the extension space excluding~$C$) into $q+1$ parts,
$C_1 , C_2 , \ldots , C_{q+1}$, each one contains $q-1$ cosets with $q^2$ matrices, i.e.~$C_j$,
$1 \leq j \leq q+1$, contains $q^2(q-1)$ matrices.

The set $\cD$ of 3-subspaces of $\F_q^4$ has size $q^3 + q^2 +q +1$.
By Lemma~\ref{lem:representDbyA} each 3-subspace of $\cD$ contains a unique
2-subspace from $\cA$. By Lemma~\ref{lem:q+1} for a given such 2-subspace $X \in \cA$ there are $q+1$
different 3-subspaces of $\F_q^4$ which contain $X$ (expanded from $X$, the first
vector is defined by the lexicographic order). Let $Y_1 , Y_2 , \ldots , Y_{q+1}$
be the $q+1$ subspaces of $\cD$ which contain $X$, where $Y_j \deff E(X,u)$ for a column vector
$u \in \F_q^4$.

For any $j$, $1 \leq j \leq q+1$, we extend the 3-subspace $Y_j$ using $C_j$ as follows.
For each $2 \times (q+1)$ matrix $Z$ from the $q^2(q-1)$ matrices of $C_j$ and for each column vector $v$ of length 2
from $\F_q^2$ we form the expanded representation $E(Z,v)$. $Y_j$ is extended with $E(Z,v)$, i.e
the new 3-subspace is represented by a $6 \times (q^2+q+1)$ matrix whose first four rows
is the matrix representation of $Y_j$ and the last two rows are $E(Z,v)$.
The result is a set $\dS_{X,j}$ which contains $q^4 (q-1)$ distinct 3-subspaces ($q^2(q-1)$ matrices in $C_j$,
where each matrix is expanded with $q^2$ vectors of length 2).

The set of all 3-subspaces formed from the 2-subspace $X \in \cA$ will be denoted by $\dS_X$
and its size is $q^4 (q-1) (q+1)=q^4 (q^2-1)$. The set of all 3-subspaces formed from $\cD$
will be denoted by $\dS_\cD$ and its size is $q^4 (q^2-1)(q^2+1) = q^4 (q^4 -1)$ since the
size of $\cA$ (from which~$X$~was taken) is $q^2+1$.

Tables~\ref{tab:S_D1} and~\ref{tab:S_D2} present first forty eight 3-subspaces of $\dS_\cD$ for $q=2$, where $X$ is taken
as the first 2-subspace of $\cA$ in Table~\ref{tab:partition} and the cosets of the extension code are taken
from Table~\ref{tab:extension}. The 3-subspaces $Y_1=\cY_1$, $Y_2=\cY_2$,
and $Y_3=\cY_3$ are presented first with their basis as in Table~\ref{tab:3_xubspaces} and
after that with their matrix representation. Note, that the first four rows of the 3-subspaces in Table~\ref{tab:S_D2}
form the matrix representation of $Y_1$, $Y_2$, and $Y_3$.
The other 192 3-subspaces of~$\dS_\cD$ are presented in Tables~\ref{tab:S_D3},~\ref{tab:S_D4},~\ref{tab:S_D5}, and~\ref{tab:S_D6}.

\begin{table*}[ht]
\centering \caption{The three 3-subspaces of the set $\cD$ which contain the first 2-subspace of $\cA$}
\label{tab:S_D1}
\begin{small}
\begin{tabular}{|c|c|c||c|c|c|}
\hline
$\cY_1$ & $\cY_2$ & $\cY_3$ & $Y_1$ & $Y_2$ & $Y_3$  \\ \hline
\hspace{-0.4cm} $\begin{array}{c}
000 \\
010 \\
011 \\
100
\end{array}$ \hspace{-0.4cm} &
\hspace{-0.4cm} $\begin{array}{c}
001 \\
010 \\
010 \\
100
\end{array}$ \hspace{-0.4cm} &
\hspace{-0.4cm} $\begin{array}{c}
001 \\
010 \\
011 \\
100
\end{array}$ \hspace{-0.4cm} &
\hspace{-0.4cm} $\begin{array}{c}
0000000 \\
0110011 \\
0111100 \\
1010101
\end{array}$ \hspace{-0.4cm} &
\hspace{-0.4cm} $\begin{array}{c}
0001111 \\
0110011 \\
0110011 \\
1010101
\end{array}$ \hspace{-0.4cm} &
\hspace{-0.4cm} $\begin{array}{c}
0001111 \\
0110011 \\
0111100 \\
1010101
\end{array}$ \hspace{-0.4cm}
\\ \hline
\end{tabular}
\end{small}
\end{table*}

Let
$$
\dS \deff \dS_0 \cup \dS_\cA \cup \dS_\cB \cup \dS_\cC  \cup \dS_\cD ~.
$$

A simple algebraic computation leads to
\begin{lemma}
\label{lem:num_sub}
$$
| \dS | =|\dS_0| + |\dS_\cA| + |\dS_\cB| + |\dS_\cC | + |\dS_\cD| = \frac{\sbinomq{7}{2}}{\sbinomq{3}{2}}~.
$$
\end{lemma}

By Lemma~\ref{lem:num_sub}, the number of subspaces in the sets $\dS$
is the same as the number of\linebreak 3-subspaces in a $q$-Fano plane. Recall, that by Lemma~\ref{lem:spread_A}
we have that $\dS_0 \cup \dS_\cA$ is a spread. Hence, to show that
$\dS \setminus (\dS_0 \cup \dS_\cA)$ is a residual $q$-Fano plane it is sufficient to prove that either each 2-subspace of $\F_q^6$
which is not contained in $\dS_0 \cup \dS_\cA$ is contained in at least $q^2$ subspaces of $\dS$, or
each 2-subspace of $\F_q^6$
which is not contained in $\dS_0 \cup \dS_\cA$ is contained in at most $q^2$ subspaces of~$\dS$.

\begin{table*}[!ht]
\centering \caption{Extensions of Type D with $Y_1=\cY_1$, $Y_2=\cY_2$, and $Y_3=\cY_3$}
\label{tab:S_D2}
\begin{small}
\begin{tabular}{|c||c|c|c|c|c|c|c|c|}
\hline
\hspace{-0.2cm} $Y_1, C_1$ \hspace{-0.2cm} & \hspace{-0.4cm} $\begin{array}{c}
0000000 \\
0110011 \\
0111100 \\
1010101 \\
0000000 \\
1100110
\end{array}$ \hspace{-0.4cm} &
\hspace{-0.4cm} $\begin{array}{c}
0000000 \\
0110011 \\
0111100 \\
1010101 \\
0000000 \\
1101001
\end{array}$ \hspace{-0.4cm} &
\hspace{-0.4cm} $\begin{array}{c}
0000000 \\
0110011 \\
0111100 \\
1010101 \\
0001111 \\
1100110
\end{array}$ \hspace{-0.4cm} &
\hspace{-0.4cm} $\begin{array}{c}
0000000 \\
0110011 \\
0111100 \\
1010101 \\
0001111 \\
1101001
\end{array}$ \hspace{-0.4cm} &
\hspace{-0.4cm} $\begin{array}{c}
0000000 \\
0110011 \\
0111100 \\
1010101 \\
0110011 \\
0000000
\end{array}$ \hspace{-0.4cm} &
\hspace{-0.4cm} $\begin{array}{c}
0000000 \\
0110011 \\
0111100 \\
1010101 \\
0110011 \\
0001111
\end{array}$ \hspace{-0.4cm} &
\hspace{-0.4cm} $\begin{array}{c}
0000000 \\
0110011 \\
0111100 \\
1010101 \\
0111100 \\
0000000
\end{array}$ \hspace{-0.4cm} &
\hspace{-0.4cm} $\begin{array}{c}
0000000 \\
0110011 \\
0111100 \\
1010101 \\
0111100 \\
0001111
\end{array}$ \hspace{-0.4cm}
\\ \hline
\hspace{-0.2cm} $Y_1, C_1$ \hspace{-0.2cm} &\hspace{-0.4cm} $\begin{array}{c}
0000000 \\
0110011 \\
0111100 \\
1010101 \\
1010101 \\
1010101
\end{array}$ \hspace{-0.4cm} &
\hspace{-0.4cm} $\begin{array}{c}
0000000 \\
0110011 \\
0111100 \\
1010101 \\
1010101 \\
1011010
\end{array}$ \hspace{-0.4cm} &
\hspace{-0.4cm} $\begin{array}{c}
0000000 \\
0110011 \\
0111100 \\
1010101 \\
1011010 \\
1010101
\end{array}$ \hspace{-0.4cm} &
\hspace{-0.4cm} $\begin{array}{c}
0000000 \\
0110011 \\
0111100 \\
1010101 \\
1011010 \\
1011010
\end{array}$ \hspace{-0.4cm} &
\hspace{-0.4cm} $\begin{array}{c}
0000000 \\
0110011 \\
0111100 \\
1010101 \\
1100110 \\
0110011
\end{array}$ \hspace{-0.4cm} &
\hspace{-0.4cm} $\begin{array}{c}
0000000 \\
0110011 \\
0111100 \\
1010101 \\
1100110 \\
0111100
\end{array}$ \hspace{-0.4cm} &
\hspace{-0.4cm} $\begin{array}{c}
0000000 \\
0110011 \\
0111100 \\
1010101 \\
1101001 \\
0110011
\end{array}$ \hspace{-0.4cm} &
\hspace{-0.4cm} $\begin{array}{c}
0000000 \\
0110011 \\
0111100 \\
1010101 \\
1101001 \\
0111100
\end{array}$ \hspace{-0.4cm}
\\ \hline
\hspace{-0.2cm} $Y_2, C_2$ \hspace{-0.2cm} &\hspace{-0.4cm} $\begin{array}{c}
0001111 \\
0110011 \\
0110011 \\
1010101 \\
1100110 \\
0000000
\end{array}$ \hspace{-0.4cm} &
\hspace{-0.4cm} $\begin{array}{c}
0001111 \\
0110011 \\
0110011 \\
1010101 \\
1100110 \\
0001111
\end{array}$ \hspace{-0.4cm} &
\hspace{-0.4cm} $\begin{array}{c}
0001111 \\
0110011 \\
0110011 \\
1010101 \\
1101001 \\
0000000
\end{array}$ \hspace{-0.4cm} &
\hspace{-0.4cm} $\begin{array}{c}
0001111 \\
0110011 \\
0110011 \\
1010101 \\
1101001 \\
0001111
\end{array}$ \hspace{-0.4cm} &
\hspace{-0.4cm} $\begin{array}{c}
0001111 \\
0110011 \\
0110011 \\
1010101 \\
1010101 \\
1100110
\end{array}$ \hspace{-0.4cm} &
\hspace{-0.4cm} $\begin{array}{c}
0001111 \\
0110011 \\
0110011 \\
1010101 \\
1010101 \\
1101001
\end{array}$ \hspace{-0.4cm} &
\hspace{-0.4cm} $\begin{array}{c}
0001111 \\
0110011 \\
0110011 \\
1010101 \\
1011010 \\
1100110
\end{array}$ \hspace{-0.4cm} &
\hspace{-0.4cm} $\begin{array}{c}
0001111 \\
0110011 \\
0110011 \\
1010101 \\
1011010 \\
1101001
\end{array}$ \hspace{-0.4cm}
\\ \hline
\hspace{-0.2cm} $Y_2, C_2$ \hspace{-0.2cm} &\hspace{-0.4cm} $\begin{array}{c}
0001111 \\
0110011 \\
0110011 \\
1010101 \\
0110011 \\
0110011
\end{array}$ \hspace{-0.4cm} &
\hspace{-0.4cm} $\begin{array}{c}
0001111 \\
0110011 \\
0110011 \\
1010101 \\
0110011 \\
0111100
\end{array}$ \hspace{-0.4cm} &
\hspace{-0.4cm} $\begin{array}{c}
0001111 \\
0110011 \\
0110011 \\
1010101 \\
0111100 \\
0110011
\end{array}$ \hspace{-0.4cm} &
\hspace{-0.4cm} $\begin{array}{c}
0001111 \\
0110011 \\
0110011 \\
1010101 \\
0111100 \\
0111100
\end{array}$ \hspace{-0.4cm} &
\hspace{-0.4cm} $\begin{array}{c}
0001111 \\
0110011 \\
0110011 \\
1010101 \\
0000000 \\
1010101
\end{array}$ \hspace{-0.4cm} &
\hspace{-0.4cm} $\begin{array}{c}
0001111 \\
0110011 \\
0110011 \\
1010101 \\
0000000 \\
1011010
\end{array}$ \hspace{-0.4cm} &
\hspace{-0.4cm} $\begin{array}{c}
0001111 \\
0110011 \\
0110011 \\
1010101 \\
0001111 \\
1010101
\end{array}$ \hspace{-0.4cm} &
\hspace{-0.4cm} $\begin{array}{c}
0001111 \\
0110011 \\
0110011 \\
1010101 \\
0001111 \\
1011010
\end{array}$ \hspace{-0.4cm}
\\ \hline
\hspace{-0.2cm} $Y_3, C_3$ \hspace{-0.2cm} &
\hspace{-0.4cm} $\begin{array}{c}
0001111 \\
0110011 \\
0111100 \\
1010101 \\
1100110 \\
1100110
\end{array}$ \hspace{-0.4cm} &
\hspace{-0.4cm} $\begin{array}{c}
0001111 \\
0110011 \\
0111100 \\
1010101 \\
1100110 \\
1101001
\end{array}$ \hspace{-0.4cm} &
\hspace{-0.4cm} $\begin{array}{c}
0001111 \\
0110011 \\
0111100 \\
1010101 \\
1101001 \\
1100110
\end{array}$ \hspace{-0.4cm} &
\hspace{-0.4cm} $\begin{array}{c}
0001111 \\
0110011 \\
0111100 \\
1010101 \\
1101001 \\
1101001
\end{array}$ \hspace{-0.4cm} &
\hspace{-0.4cm} $\begin{array}{c}
0001111 \\
0110011 \\
0111100 \\
1010101 \\
1010101 \\
0000000
\end{array}$ \hspace{-0.4cm} &
\hspace{-0.4cm} $\begin{array}{c}
0001111 \\
0110011 \\
0111100 \\
1010101 \\
1010101 \\
0001111
\end{array}$ \hspace{-0.4cm} &
\hspace{-0.4cm} $\begin{array}{c}
0001111 \\
0110011 \\
0111100 \\
1010101 \\
1011010 \\
0000000
\end{array}$ \hspace{-0.4cm} &
\hspace{-0.4cm} $\begin{array}{c}
0001111 \\
0110011 \\
0111100 \\
1010101 \\
1011010 \\
0001111
\end{array}$ \hspace{-0.4cm}
\\ \hline
\hspace{-0.2cm} $Y_3, C_3$ \hspace{-0.2cm} &
\hspace{-0.4cm} $\begin{array}{c}
0001111 \\
0110011 \\
0111100 \\
1010101 \\
0110011 \\
1010101
\end{array}$ \hspace{-0.4cm} &
\hspace{-0.4cm} $\begin{array}{c}
0001111 \\
0110011 \\
0111100 \\
1010101 \\
0110011 \\
1011010
\end{array}$ \hspace{-0.4cm} &
\hspace{-0.4cm} $\begin{array}{c}
0001111 \\
0110011 \\
0111100 \\
1010101 \\
0111100 \\
1010101
\end{array}$ \hspace{-0.4cm} &
\hspace{-0.4cm} $\begin{array}{c}
0001111 \\
0110011 \\
0111100 \\
1010101 \\
0111100 \\
1011010
\end{array}$ \hspace{-0.4cm} &
\hspace{-0.4cm} $\begin{array}{c}
0001111 \\
0110011 \\
0111100 \\
1010101 \\
0000000 \\
0110011
\end{array}$ \hspace{-0.4cm} &
\hspace{-0.4cm} $\begin{array}{c}
0001111 \\
0110011 \\
0111100 \\
1010101 \\
0000000 \\
0111100
\end{array}$ \hspace{-0.4cm} &
\hspace{-0.4cm} $\begin{array}{c}
0001111 \\
0110011 \\
0111100 \\
1010101 \\
0001111 \\
0110011
\end{array}$ \hspace{-0.4cm} &
\hspace{-0.4cm} $\begin{array}{c}
0001111 \\
0110011 \\
0111100 \\
1010101 \\
0001111 \\
0111100
\end{array}$ \hspace{-0.4cm}
\\ \hline
\end{tabular}
\end{small}
\end{table*}

\begin{lemma}
\label{lem:AinD}
Each 2-subspace of $\F_q^6$ which can be extended from a 2-subspace of $\cA$, but not extended to a 2-subspace of the spread
$\dS_0 \cup \dS_\cA$, is contained $q^2$ times in the 3-subspaces of~$\dS_\cD$.
\end{lemma}
\begin{proof}
Since the 2-subspaces of $\cA$ are extended only to the spread of $\dS_0 \cup \dS_\cA$,
it follows that
any 2-subspaces of $\F_q^6$ which can be extended from a 2-subspace of $\cA$ was
formed by extending the 3-subspaces of $\dS_\cD$.
By Lemma~\ref{lem:representDbyA} each 2-subspace of $\cA$ is contained in $q+1$
distinct 3-subspaces of $\cD$.
By the definition for the representation of 3-subspaces, this 2-subspace appear in the first four rows
and the first $q+1$ column of the 3-subspace representation. Let $X \in \cA$ and let
$Y_1 , Y_2 , \ldots , Y_{q+1}$ be the $q+1$ subspaces of $\cD$ which contain $X$.
By the extensions of $\cD$, each matrix $Z$ of the extension space, which is not part
of the extension code, is used $q^2$ times to extend $X$, using the $2 \times (q^2+q+1)$
matrices $E(Z,u)$, where any column vector of length 2 over $\F_q$ is used once as $u$.
By lemma~\ref{lem:q4extensions}, these are all the possible extensions of 2-subspaces from $\cA$
(note, that the extensions of subspaces from $\cA$ with the extension code are exactly the 2-subspaces of $\dS_\cA$.).
\end{proof}

\begin{lemma}
\label{lem:BinB}
Each 2-subspaces of $\F_q^6$ extended from a 2-subspace of $\cB$ is contained exactly once
in the 3-subspaces of $\dS_\cB$.
\end{lemma}
\begin{proof}
Any 2-subspace $X$ of $\cB$ is first extended in all the $q^2$ possible distinct ways
to a 2-subspace of $\F_q^5$. Each 2-subspace $Y$ of these $q^2$ subspaces is extended in
a unique way to a 3-subspace $Z$. Such a 3-subspace $Z$ contains all the $q^2$ distinct 2-subspaces of $\F_q^6$,
extended from~$Y$.
\end{proof}

\begin{lemma}
\label{lem:CinC}
Each 2-subspaces of $\F_q^6$ extended from a 2-subspace of $\cC$ is contained exactly once
in the 3-subspaces of $\dS_\cC$.
\end{lemma}
\begin{proof}
Any 2-subspace $X$ of $\cC$ is first extended in a unique way to a 3-subspace $Z$ of $\F_q^5$.
Such a 3-subspace $Z$ contains all the $q^2$ distinct 2-subspaces of $\F_q^5$, extended from $X$.
Each such 3-subspace $Y$ is extended to $q^2$ (out of the $q^3$)
distinct 3-subspaces of $\F_q^6$. All these $q^2$ distinct 3-subspaces have the same symbol in the last row of the
$(q+2)$-th column, in the matrix representation,
which implies that each distinct 2-subspace of $\F_q^5$ is extended in $q^2$ distinct ways
to all possible distinct 2-subspaces of $\F_q^6$ extended from $X$.
\end{proof}

For the next set of 2-subspaces we need one property of the extension space.
\begin{lemma}
\label{lem:diffExtent}
Let $C'$ be a coset of the extension code, let $M_1$ and $M_2$ two $2 \times (q+1)$
matrices of $C'$, and let $u_1$ and $u_2$ two column vectors of $\F_q^2$.
Let $\{ i_1, i_2,\ldots ,i_{q+1} \} \in \cP$, defined in Lemma~\ref{lem:exact2in3}, and let $X$ be a 3-subspace.
If $Y_1$ and $Y_2$ are extensions of $X$ with $E(M_1,u_1)$ and
$E(M_2,u_2)$, respectively, then columns $i_1, i_2,\ldots ,i_{q+1}$ of $Y_1$ and $Y_2$
define two different 2-subspaces of $\F_q^6$ unless $M_1=M_2$ and $\{ i_1, i_2,\ldots ,i_{q+1} \}=\{ 1,2,\ldots,q+1\}$
or ${M_1=M_2}$ and $u_1=u_2$.
\end{lemma}
\begin{proof}
By the definition of the extension code, the columns of $M_1$ and $M_2$ are distinct in pairs unless
$M_1=M_2$. Hence, by Lemma~\ref{lem:2in3} we infer the result in the case that ${M_1 \neq M_2}$.
If ${M_1 = M_2}$ then $u_1 \neq u_2$ implies that except for the first $q+1$ columns
all the columns of $E(M_1,u_1)$ and $E(M_2,u_2)$ are different in pairs and hence the result
follows from Lemma~\ref{lem:2in3}.
\end{proof}

Since each 2-subspace of either $\cB$ or $\cC$ is contained in $q^2-1$ 3-subspaces of $\cD$,
it follows as a consequence of Lemma~\ref{lem:diffExtent} that
\begin{lemma}
\label{lem:BCinD}
Each 2-subspaces of $\F_q^6$ extended from a 2-subspace of either $\cB$ or $\cC$ is contained exactly $q^2-1$
times in the 3-subspaces of $\dS_\cD$.
\end{lemma}

\begin{lemma}
\label{lem:5zero}
Each 2-subspace of $\F_q^6$ which contains a vector which start with five \emph{zeroes} is contained either
in $\dS_0$ or contained $q^2$ times in $\dS_\cB$.
\end{lemma}
\begin{proof}
The unique 2-subspace in which all vectors start with four or five \emph{zeroes} is contained in $\dS_0$.

Vectors which start with five \emph{zeroes} are contained in $\dS_0$ and in the extensions of 2-subspaces
from $\cB$. The reason is that the 2-subspaces of $\cB$ are first extended to 2-subspaces of $\F_q^5$ in $q^2$
distinct ways. Since $\cB$ contains $q$ spreads, it follows that each nonzero vector of length 4 is contained $q$ times in the 2-subspaces
of $\cB$. Since there are $q^2$ distinct extensions of a 2-subspace of $\F_q^4$ to a 2-subspace of $\F_q^5$ is follows
that each vector of length 4 is extended with a symbol $\xi \in \F_q$ to a vector of length 5 exactly $q$ times.
Hence, each nonzero vector of length 5 appears in the extensions of $\cB$ to 2-subspaces of $\F_q^5$ exactly $q^2$ times.
Thus, each vector of length 6 appears exactly $q^2$ times in the extensions
(with expansions) of $\cB$ to 3-subspaces in a unique way. Thus, each 2-subspace
which contains a vector of length 6 starting with 5 \emph{zeroes} is contained in $q^2$ distinct 3-subspaces of $\dS_\cB$.
\end{proof}

\begin{lemma}
\label{lem:4zero}
Each 2-subspace of $\F_q^6$ which contains a vector which start with four \emph{zeroes} and the
5-th symbol is nonzero, is contained either
in $\dS_0$ or contained $q^2$ times in $\dS_\cC$.
\end{lemma}
\begin{proof}
The unique 2-subspace in which all vectors start with four \emph{zeroes} is contained in $\dS_0$.

Vectors which start with four \emph{zeroes} and the 5-th symbol is nonzero,
are contained in $\dS_0$ and in the extensions of subspaces
from $\cC$. Since $\cC$ contains $q^2$ spreads, it follows that each nonzero vector of
length 4 is contained $q^2$ times in the 2-subspaces of $\cC$.
Hence, each vector of length 5 appears exactly $q^2$ times in the extensions
(with expansions) of $\cC$ to 3-subspaces in a unique way. Thus, each 2-subspace
which contains a vector of length 5 starting with 4 \emph{zeroes} and 5-th nonzero,
is contained in $q^2$ distinct 3-subspaces of $\F_q^5$ extended (and expanded) from $\cC$.
For each 2-subspace of $\F_q^5$ which contains a vector which starts with 4 \emph{zeroes}
there are $q^2$ distinct extensions to a 2-subspace of $\F_q^6$. Each one is considered in the
extensions of $\cC_\xi$, $\xi \in \F_q$, and since each 2-subspace of $\F_q^5$ was contained
$q^2$ times, it follows that the same is true for the 2-subspaces of $\F_q^6$ which
contain a vector which start with four \emph{zeroes}and the
5-th symbol is nonzero.
\end{proof}

A consequence of Lemmas~\ref{lem:num_sub},~\ref{lem:AinD},~\ref{lem:BinB},~\ref{lem:CinC},~\ref{lem:BCinD},~\ref{lem:5zero},~\ref{lem:4zero},
we have the concluding result.

\begin{theorem}
$\dS_0 \cup \dS_\cA$ is a derived $q$-Fano plane and $\dS_\cB \cup \dS_\cC \cup \dS_\cD$ is a residual $q$-Fano plane.
\end{theorem}

\section{Conclusions and Future Research}
\label{sec:conclude}

We have presented a new definition for the residual $q$-design which reflects better the
relations between the design on one side and its derived design and residual designs on the other hand.
We have constructed designs with the parameters of the residual design of the\linebreak $q$-Fano plane
for each power of a prime $q$. This is the closest as was achieved until today towards a construction
of infinite family of $q$-Steiner systems, arguably, the most intriguing open problem in block design today.
Our construction is flexible which enable to construct many residual $q$-Fano planes for each $q$. The number
of different residual $q$-Fano planes is increased with the increase of $q$. The first point with flexibility is
the number of parallelisms in $\F_q^4$ which are generally increasing as $q$ get larger. The number of
partitions of the spreads in such a parallelism into the sets $\cA$, $\cB$, and $\cC$, is clearly
increasing as $q$ get larger. Similarly, $\cC$ can be partitions in a few different ways to
$\{ \cC_\xi ~:~ \xi \in \F_q \}$ and the number of such partitions is clearly increasing with $q$.
The extension code can be chosen in a few different ways and the number
of different ways is also increasing when $q$ increases. Finally, there are many different ways to make
the extensions of Type D. First, the cosets of the extension code (the extension space without the extension code)
can be partitioned in a few different ways (with an exception for $q=2$) to $C_1,C_2,\ldots,C_{q+1}$
and these number of different ways is clearly increasing with the increase of $q$.
The matching of the pairs $(Y_i,C_i)$, for the extension of Type D, can be done in $(q+1)!$ different ways and this can be done
for each spread in $\cA$. Hence, we have many different residual $q$-Fano planes for each $q$ and each one
might have different properties and can be used for different purpose. This is a subject for future research.
In particular one can find different residual $q$-Fano planes which differ in a small number of subspaces
(by using pairs in the extensions of Type D which differ only in one transposition). One can easily
verify that the structure obtained from the dual subspaces of the subspaces in a residual $q$-Fano plane is also
a residual $q$-Fano plane. This can lead to other interesting properties of the $q$-Fano plane and
this is a topic for future research. Finally, an applications of
the new structure in network coding is presented in~\cite{EtZh17}.

The new construction and the new structure open also a sequence of other directions for future research, for which we list a few:

\begin{itemize}
\item Provide more constructions for residual $q$-Steiner systems with other parameters.

\item Can a residual $q$-Steiner system exists, while a related $q$-Steiner system
does not exist? We conjecture that the answer is positive.

\item Prove that the residual $q$-Fano plane constructed can be extended or cannot be extended
to a $q$-Fano plane. We conjecture that for $q=2$ it cannot be extended, while for some $q>2$ such
an extension might be possible.

\item Examine the properties of the residual $q$-Steiner systems with respect to the possible
existence of a related $q$-Steiner systems.
\end{itemize}

Finally, we note that the subspaces used throughout the construction can be represented by their
basis and the same is true for the construction. We believe that with such more natural representation
the proof of the main result and its verification will be more complicated and less intuitive.
But, the construction can be easily given with basis for subspaces. For 2-subspaces the first
columns can be taken as the basis. For 3-subspaces, the $(q+2)$-th column can be taken to complete
the basis. These three columns for the basis are well defined and hence one can generated the subspaces
of the design without generating the matrices.

\begin{center}
{\bf Acknowledgments}
\end{center}
Tuvi Etzion would like to thank Alex Vardy for enormous conversations on
the problem during the last ten years.

\begin{appendix}

\begin{center}
{\bf Appendix}
\end{center}

\begin{table}[ht]
\centering \caption{Extensions of Type D with $Y_1=\cY_4$, $Y_2=\cY_5$, and $Y_3=\cY_6$}
\label{tab:S_D3}
\begin{tiny}
\begin{tabular}{|c|c|c|c|c|c|c|c|c|c|c|c|}
\hline
\hspace{-0.4cm} $\begin{array}{c}
0110011 \\
0111100 \\
1010101 \\
0000000 \\
0000000 \\
1100110
\end{array}$ \hspace{-0.4cm}  & \hspace{-0.4cm} $\begin{array}{c}
0110011 \\
0111100 \\
1010101 \\
0000000 \\
0000000 \\
1101001
\end{array}$ \hspace{-0.4cm} & \hspace{-0.4cm} $\begin{array}{c}
0110011 \\
0111100 \\
1010101 \\
0000000 \\
0001111 \\
1100110
\end{array}$ \hspace{-0.4cm}  & \hspace{-0.4cm} $\begin{array}{c}
0110011 \\
0111100 \\
1010101 \\
0000000 \\
0001111 \\
1101001
\end{array}$ \hspace{-0.4cm} & \hspace{-0.4cm} $\begin{array}{c}
0110011 \\
0111100 \\
1010101 \\
0000000 \\
0110011 \\
0000000
\end{array}$ \hspace{-0.4cm} & \hspace{-0.4cm} $\begin{array}{c}
0110011 \\
0111100 \\
1010101 \\
0000000 \\
0110011 \\
0001111
\end{array}$ \hspace{-0.4cm} & \hspace{-0.4cm} $\begin{array}{c}
0110011 \\
0111100 \\
1010101 \\
0000000 \\
0111100 \\
0000000
\end{array}$ \hspace{-0.4cm} & \hspace{-0.4cm} $\begin{array}{c}
0110011 \\
0111100 \\
1010101 \\
0000000 \\
0111100 \\
0001111
\end{array}$ \hspace{-0.4cm} & \hspace{-0.4cm} $\begin{array}{c}
0110011 \\
0111100 \\
1010101 \\
0000000 \\
1010101 \\
1010101
\end{array}$ \hspace{-0.4cm} & \hspace{-0.4cm} $\begin{array}{c}
0110011 \\
0111100 \\
1010101 \\
0000000 \\
1010101 \\
1011010
\end{array}$ \hspace{-0.4cm} & \hspace{-0.4cm} $\begin{array}{c}
0110011 \\
0111100 \\
1010101 \\
0000000 \\
1011010 \\
1010101
\end{array}$ \hspace{-0.4cm} &  \hspace{-0.4cm} $\begin{array}{c}
0110011 \\
0111100 \\
1010101 \\
0000000 \\
1011010 \\
1011010
\end{array}$ \hspace{-0.4cm}
\\ \hline
\hspace{-0.4cm} $\begin{array}{c}
0110011 \\
0111100 \\
1010101 \\
0000000 \\
1100110 \\
0110011
\end{array}$ \hspace{-0.4cm} & \hspace{-0.4cm} $\begin{array}{c}
0110011 \\
0111100 \\
1010101 \\
0000000 \\
1100110 \\
0111100
\end{array}$ \hspace{-0.4cm} & \hspace{-0.4cm} $\begin{array}{c}
0110011 \\
0111100 \\
1010101 \\
0000000 \\
1101001 \\
0110011
\end{array}$ \hspace{-0.4cm} & \hspace{-0.4cm} $\begin{array}{c}
0110011 \\
0111100 \\
1010101 \\
0000000 \\
1101001 \\
0111100
\end{array}$ \hspace{-0.4cm} & \hspace{-0.4cm} $\begin{array}{c}
0110011 \\
0110011 \\
1010101 \\
0001111 \\
1100110 \\
0000000
\end{array}$ \hspace{-0.4cm} & \hspace{-0.4cm} $\begin{array}{c}
0110011 \\
0110011 \\
1010101 \\
0001111 \\
1100110 \\
0001111
\end{array}$ \hspace{-0.4cm} & \hspace{-0.4cm} $\begin{array}{c}
0110011 \\
0110011 \\
1010101 \\
0001111 \\
1101001 \\
0000000
\end{array}$ \hspace{-0.4cm} & \hspace{-0.4cm} $\begin{array}{c}
0110011 \\
0110011 \\
1010101 \\
0001111 \\
1101001 \\
0001111
\end{array}$ \hspace{-0.4cm} & \hspace{-0.4cm} $\begin{array}{c}
0110011 \\
0110011 \\
1010101 \\
0001111 \\
1010101 \\
1100110
\end{array}$ \hspace{-0.4cm} & \hspace{-0.4cm} $\begin{array}{c}
0110011 \\
0110011 \\
1010101 \\
0001111 \\
1010101 \\
1101001
\end{array}$ \hspace{-0.4cm} & \hspace{-0.4cm} $\begin{array}{c}
0110011 \\
0110011 \\
1010101 \\
0001111 \\
1011010 \\
1100110
\end{array}$ \hspace{-0.4cm} & \hspace{-0.4cm} $\begin{array}{c}
0110011 \\
0110011 \\
1010101 \\
0001111 \\
1011010 \\
1101001
\end{array}$ \hspace{-0.4cm}
\\ \hline
\hspace{-0.4cm} $\begin{array}{c}
0110011 \\
0110011 \\
1010101 \\
0001111 \\
0110011 \\
0110011
\end{array}$ \hspace{-0.4cm} & \hspace{-0.4cm} $\begin{array}{c}
0110011 \\
0110011 \\
1010101 \\
0001111 \\
0110011 \\
0111100
\end{array}$ \hspace{-0.4cm} & \hspace{-0.4cm} $\begin{array}{c}
0110011 \\
0110011 \\
1010101 \\
0001111 \\
0111100 \\
0110011
\end{array}$ \hspace{-0.4cm} & \hspace{-0.4cm} $\begin{array}{c}
0110011 \\
0110011 \\
1010101 \\
0001111 \\
0111100 \\
0111100
\end{array}$ \hspace{-0.4cm} & \hspace{-0.4cm} $\begin{array}{c}
0110011 \\
0110011 \\
1010101 \\
0001111 \\
0000000 \\
1010101
\end{array}$ \hspace{-0.4cm} & \hspace{-0.4cm} $\begin{array}{c}
0110011 \\
0110011 \\
1010101 \\
0001111 \\
0000000 \\
1011010
\end{array}$ \hspace{-0.4cm} & \hspace{-0.4cm} $\begin{array}{c}
0110011 \\
0110011 \\
1010101 \\
0001111 \\
0001111 \\
1010101
\end{array}$ \hspace{-0.4cm} & \hspace{-0.4cm} $\begin{array}{c}
0110011 \\
0110011 \\
1010101 \\
0001111 \\
0001111 \\
1011010
\end{array}$ \hspace{-0.4cm} & \hspace{-0.4cm} $\begin{array}{c}
0110011 \\
0111100 \\
1010101 \\
0001111 \\
1100110 \\
1100110
\end{array}$ \hspace{-0.4cm} & \hspace{-0.4cm} $\begin{array}{c}
0110011 \\
0111100 \\
1010101 \\
0001111 \\
1100110 \\
1101001
\end{array}$ \hspace{-0.4cm} &  \hspace{-0.4cm} $\begin{array}{c}
0110011 \\
0111100 \\
1010101 \\
0001111 \\
1101001 \\
1100110
\end{array}$ \hspace{-0.4cm} & \hspace{-0.4cm} $\begin{array}{c}
0110011 \\
0111100 \\
1010101 \\
0001111 \\
1101001 \\
1101001
\end{array}$ \hspace{-0.4cm}
\\ \hline
\hspace{-0.4cm} $\begin{array}{c}
0110011 \\
0111100 \\
1010101 \\
0001111 \\
1010101 \\
0000000
\end{array}$ \hspace{-0.4cm} & \hspace{-0.4cm} $\begin{array}{c}
0110011 \\
0111100 \\
1010101 \\
0001111 \\
1010101 \\
0001111
\end{array}$ \hspace{-0.4cm} & \hspace{-0.4cm} $\begin{array}{c}
0110011 \\
0111100 \\
1010101 \\
0001111 \\
1011010 \\
0000000
\end{array}$ \hspace{-0.4cm} & \hspace{-0.4cm} $\begin{array}{c}
0110011 \\
0111100 \\
1010101 \\
0001111 \\
1011010 \\
0001111
\end{array}$ \hspace{-0.4cm} & \hspace{-0.4cm} $\begin{array}{c}
0110011 \\
0111100 \\
1010101 \\
0001111 \\
0110011 \\
1010101
\end{array}$ \hspace{-0.4cm} & \hspace{-0.4cm} $\begin{array}{c}
0110011 \\
0111100 \\
1010101 \\
0001111 \\
0110011 \\
1011010
\end{array}$ \hspace{-0.4cm} & \hspace{-0.4cm} $\begin{array}{c}
0110011 \\
0111100 \\
1010101 \\
0001111 \\
0111100 \\
1010101
\end{array}$ \hspace{-0.4cm} & \hspace{-0.4cm} $\begin{array}{c}
0110011 \\
0111100 \\
1010101 \\
0001111 \\
0111100 \\
1011010
\end{array}$ \hspace{-0.4cm} & \hspace{-0.4cm} $\begin{array}{c}
0110011 \\
0111100 \\
1010101 \\
0001111 \\
0000000 \\
0110011
\end{array}$ \hspace{-0.4cm} & \hspace{-0.4cm} $\begin{array}{c}
0110011 \\
0111100 \\
1010101 \\
0001111 \\
0000000 \\
0111100
\end{array}$ \hspace{-0.4cm} & \hspace{-0.4cm} $\begin{array}{c}
0110011 \\
0111100 \\
1010101 \\
0001111 \\
0001111 \\
0110011
\end{array}$ \hspace{-0.4cm} & \hspace{-0.4cm} $\begin{array}{c}
0110011 \\
0111100 \\
1010101 \\
0001111 \\
0001111 \\
0111100
\end{array}$ \hspace{-0.4cm}
\\ \hline
\end{tabular}
\end{tiny}
\end{table}

\begin{table}[ht]
\centering \caption{Extensions of Type D with $Y_1=\cY_7$, $Y_2=\cY_8$, and $Y_3=\cY_9$}
\label{tab:S_D4}
\begin{tiny}
\begin{tabular}{|c|c|c|c|c|c|c|c|c|c|c|c|}
\hline
\hspace{-0.4cm} $\begin{array}{c}
0110011 \\
1010101 \\
0110011 \\
0111100 \\
0000000 \\
1100110
\end{array}$ \hspace{-0.4cm} & \hspace{-0.4cm} $\begin{array}{c}
0110011 \\
1010101 \\
0110011 \\
0111100 \\
0000000 \\
1101001
\end{array}$ \hspace{-0.4cm} & \hspace{-0.4cm} $\begin{array}{c}
0110011 \\
1010101 \\
0110011 \\
0111100 \\
0001111 \\
1100110
\end{array}$ \hspace{-0.4cm} & \hspace{-0.4cm} $\begin{array}{c}
0110011 \\
1010101 \\
0110011 \\
0111100 \\
0001111 \\
1101001
\end{array}$ \hspace{-0.4cm} & \hspace{-0.4cm} $\begin{array}{c}
0110011 \\
1010101 \\
0110011 \\
0111100 \\
0110011 \\
0000000
\end{array}$ \hspace{-0.4cm} & \hspace{-0.4cm} $\begin{array}{c}
0110011 \\
1010101 \\
0110011 \\
0111100 \\
0110011 \\
0001111
\end{array}$ \hspace{-0.4cm} & \hspace{-0.4cm} $\begin{array}{c}
0110011 \\
1010101 \\
0110011 \\
0111100 \\
0111100 \\
0000000
\end{array}$ \hspace{-0.4cm} & \hspace{-0.4cm} $\begin{array}{c}
0110011 \\
1010101 \\
0110011 \\
0111100 \\
0111100 \\
0001111
\end{array}$ \hspace{-0.4cm} & \hspace{-0.4cm} $\begin{array}{c}
0110011 \\
1010101 \\
0110011 \\
0111100 \\
1010101 \\
1010101
\end{array}$ \hspace{-0.4cm} & \hspace{-0.4cm} $\begin{array}{c}
0110011 \\
1010101 \\
0110011 \\
0111100 \\
1010101 \\
1011010
\end{array}$ \hspace{-0.4cm} & \hspace{-0.4cm} $\begin{array}{c}
0110011 \\
1010101 \\
0110011 \\
0111100 \\
1011010 \\
1010101
\end{array}$ \hspace{-0.4cm} & \hspace{-0.4cm} $\begin{array}{c}
0110011 \\
1010101 \\
0110011 \\
0111100 \\
1011010 \\
1011010
\end{array}$ \hspace{-0.4cm}
\\ \hline
\hspace{-0.4cm} $\begin{array}{c}
0110011 \\
1010101 \\
0110011 \\
0111100 \\
1100110 \\
0110011
\end{array}$ \hspace{-0.4cm} & \hspace{-0.4cm} $\begin{array}{c}
0110011 \\
1010101 \\
0110011 \\
0111100 \\
1100110 \\
0111100
\end{array}$ \hspace{-0.4cm} & \hspace{-0.4cm} $\begin{array}{c}
0110011 \\
1010101 \\
0110011 \\
0111100 \\
1101001 \\
0110011
\end{array}$ \hspace{-0.4cm} & \hspace{-0.4cm} $\begin{array}{c}
0110011 \\
1010101 \\
0110011 \\
0111100 \\
1101001 \\
0111100
\end{array}$ \hspace{-0.4cm} & \hspace{-0.4cm} $\begin{array}{c}
0110011 \\
1010101 \\
0111100 \\
0111100 \\
1100110 \\
0000000
\end{array}$ \hspace{-0.4cm} & \hspace{-0.4cm} $\begin{array}{c}
0110011 \\
1010101 \\
0111100 \\
0111100 \\
1100110 \\
0001111
\end{array}$ \hspace{-0.4cm} & \hspace{-0.4cm} $\begin{array}{c}
0110011 \\
1010101 \\
0111100 \\
0111100 \\
1101001 \\
0000000
\end{array}$ \hspace{-0.4cm} & \hspace{-0.4cm} $\begin{array}{c}
0110011 \\
1010101 \\
0111100 \\
0111100 \\
1101001 \\
0001111
\end{array}$ \hspace{-0.4cm} & \hspace{-0.4cm} $\begin{array}{c}
0110011 \\
1010101 \\
0111100 \\
0111100 \\
1010101 \\
1100110
\end{array}$ \hspace{-0.4cm} & \hspace{-0.4cm} $\begin{array}{c}
0110011 \\
1010101 \\
0111100 \\
0111100 \\
1010101 \\
1101001
\end{array}$ \hspace{-0.4cm} & \hspace{-0.4cm} $\begin{array}{c}
0110011 \\
1010101 \\
0111100 \\
0111100 \\
1011010 \\
1100110
\end{array}$ \hspace{-0.4cm} & \hspace{-0.4cm} $\begin{array}{c}
0110011 \\
1010101 \\
0111100 \\
0111100 \\
1011010 \\
1101001
\end{array}$ \hspace{-0.4cm}
\\ \hline
\hspace{-0.4cm} $\begin{array}{c}
0110011 \\
1010101 \\
0111100 \\
0111100 \\
0110011 \\
0110011
\end{array}$ \hspace{-0.4cm} & \hspace{-0.4cm} $\begin{array}{c}
0110011 \\
1010101 \\
0111100 \\
0111100 \\
0110101 \\
0111100
\end{array}$ \hspace{-0.4cm} & \hspace{-0.4cm} $\begin{array}{c}
0110011 \\
1010101 \\
0111100 \\
0111100 \\
0111100 \\
0110110
\end{array}$ \hspace{-0.4cm} & \hspace{-0.4cm} $\begin{array}{c}
0110011 \\
1010101 \\
0111100 \\
0111100 \\
0111100 \\
0111100
\end{array}$ \hspace{-0.4cm} & \hspace{-0.4cm} $\begin{array}{c}
0110011 \\
1010101 \\
0111100 \\
0111100 \\
0000000 \\
1010101
\end{array}$ \hspace{-0.4cm} & \hspace{-0.4cm} $\begin{array}{c}
0110011 \\
1010101 \\
0111100 \\
0111100 \\
0000000 \\
1011010
\end{array}$ \hspace{-0.4cm} & \hspace{-0.4cm} $\begin{array}{c}
0110011 \\
1010101 \\
0111100 \\
0111100 \\
0001111 \\
1010101
\end{array}$ \hspace{-0.4cm} & \hspace{-0.4cm} $\begin{array}{c}
0110011 \\
1010101 \\
0111100 \\
0111100 \\
0001111 \\
1011010
\end{array}$ \hspace{-0.4cm} & \hspace{-0.4cm} $\begin{array}{c}
0110011 \\
1010101 \\
0111100 \\
0110011 \\
1100110 \\
1100110
\end{array}$ \hspace{-0.4cm} & \hspace{-0.4cm} $\begin{array}{c}
0110011 \\
1010101 \\
0111100 \\
0110011 \\
1100110 \\
1101001
\end{array}$ \hspace{-0.4cm} & \hspace{-0.4cm} $\begin{array}{c}
0110011 \\
1010101 \\
0111100 \\
0110011 \\
1101001 \\
1100110
\end{array}$ \hspace{-0.4cm} & \hspace{-0.4cm} $\begin{array}{c}
0110011 \\
1010101 \\
0111100 \\
0110011 \\
1101001 \\
1101001
\end{array}$ \hspace{-0.4cm}
\\ \hline
\hspace{-0.4cm} $\begin{array}{c}
0110011 \\
1010101 \\
0111100 \\
0110011 \\
1010101 \\
0000000
\end{array}$ \hspace{-0.4cm} & \hspace{-0.4cm} $\begin{array}{c}
0110011 \\
1010101 \\
0111100 \\
0110011 \\
1010101 \\
0001111
\end{array}$ \hspace{-0.4cm} & \hspace{-0.4cm} $\begin{array}{c}
0110011 \\
1010101 \\
0111100 \\
0110011 \\
1011010 \\
0000000
\end{array}$ \hspace{-0.4cm} & \hspace{-0.4cm} $\begin{array}{c}
0110011 \\
1010101 \\
0111100 \\
0110011 \\
1011010 \\
0001111
\end{array}$ \hspace{-0.4cm} & \hspace{-0.4cm} $\begin{array}{c}
0110011 \\
1010101 \\
0111100 \\
0110011 \\
0110011 \\
1010101
\end{array}$ \hspace{-0.4cm} &  \hspace{-0.4cm} $\begin{array}{c}
0110011 \\
1010101 \\
0111100 \\
0110011 \\
0110011 \\
1011010
\end{array}$ \hspace{-0.4cm} & \hspace{-0.4cm} $\begin{array}{c}
0110011 \\
1010101 \\
0111100 \\
0110011 \\
0111100 \\
1010101
\end{array}$ \hspace{-0.4cm} & \hspace{-0.4cm} $\begin{array}{c}
0110011 \\
1010101 \\
0111100 \\
0110011 \\
0111100 \\
1011010
\end{array}$ \hspace{-0.4cm} & \hspace{-0.4cm} $\begin{array}{c}
0110011 \\
1010101 \\
0111100 \\
0110011 \\
0000000 \\
0110011
\end{array}$ \hspace{-0.4cm} & \hspace{-0.4cm} $\begin{array}{c}
0110011 \\
1010101 \\
0111100 \\
0110011 \\
0000000 \\
0111100
\end{array}$ \hspace{-0.4cm} & \hspace{-0.4cm} $\begin{array}{c}
0110011 \\
1010101 \\
0111100 \\
0110011 \\
0001111 \\
0110011
\end{array}$ \hspace{-0.4cm} & \hspace{-0.4cm} $\begin{array}{c}
0110011 \\
1010101 \\
0111100 \\
0110011 \\
0001111 \\
0111100
\end{array}$ \hspace{-0.4cm}
\\ \hline
\end{tabular}
\end{tiny}
\end{table}

\begin{table}[ht]
\centering \caption{Extensions of Type D with $Y_1=\cY_{10}$, $Y_2=\cY_{11}$, and $Y_3=\cY_{12}$}
\label{tab:S_D5}
\begin{tiny}
\begin{tabular}{|c|c|c|c|c|c|c|c|c|c|c|c|}
\hline
\hspace{-0.4cm} $\begin{array}{c}
0110011 \\
1010101 \\
0000000 \\
1011010 \\
0000000 \\
1100110
\end{array}$ \hspace{-0.4cm} & \hspace{-0.4cm} $\begin{array}{c}
0110011 \\
1010101 \\
0000000 \\
1011010 \\
0000000 \\
1101001
\end{array}$ \hspace{-0.4cm} & \hspace{-0.4cm} $\begin{array}{c}
0110011 \\
1010101 \\
0000000 \\
1011010 \\
0001111 \\
1100110
\end{array}$ \hspace{-0.4cm} & \hspace{-0.4cm} $\begin{array}{c}
0110011 \\
1010101 \\
0000000 \\
1011010 \\
0001111 \\
1101001
\end{array}$ \hspace{-0.4cm} & \hspace{-0.4cm} $\begin{array}{c}
0110011 \\
1010101 \\
0000000 \\
1011010 \\
0110011 \\
0000000
\end{array}$ \hspace{-0.4cm} & \hspace{-0.4cm} $\begin{array}{c}
0110011 \\
1010101 \\
0000000 \\
1011010 \\
0110011 \\
0001111
\end{array}$ \hspace{-0.4cm} & \hspace{-0.4cm} $\begin{array}{c}
0110011 \\
1010101 \\
0000000 \\
1011010 \\
0111100 \\
0000000
\end{array}$ \hspace{-0.4cm} & \hspace{-0.4cm} $\begin{array}{c}
0110011 \\
1010101 \\
0000000 \\
1011010 \\
0111100 \\
0001111
\end{array}$ \hspace{-0.4cm} & \hspace{-0.4cm} $\begin{array}{c}
0110011 \\
1010101 \\
0000000 \\
1011010 \\
1010101 \\
1010101
\end{array}$ \hspace{-0.4cm} & \hspace{-0.4cm} $\begin{array}{c}
0110011 \\
1010101 \\
0000000 \\
1011010 \\
1010101 \\
1011010
\end{array}$ \hspace{-0.4cm} & \hspace{-0.4cm} $\begin{array}{c}
0110011 \\
1010101 \\
0000000 \\
1011010 \\
1011010 \\
1010101
\end{array}$ \hspace{-0.4cm} & \hspace{-0.4cm} $\begin{array}{c}
0110011 \\
1010101 \\
0000000 \\
1011010 \\
1011010 \\
1011010
\end{array}$ \hspace{-0.4cm}
\\ \hline
\hspace{-0.4cm} $\begin{array}{c}
0110011 \\
1010101 \\
0000000 \\
1011010 \\
1100110 \\
0110011
\end{array}$ \hspace{-0.4cm} & \hspace{-0.4cm} $\begin{array}{c}
0110011 \\
1010101 \\
0000000 \\
1011010 \\
1100110 \\
0111100
\end{array}$ \hspace{-0.4cm} & \hspace{-0.4cm} $\begin{array}{c}
0110011 \\
1010101 \\
0000000 \\
1011010 \\
1101001 \\
0110011
\end{array}$ \hspace{-0.4cm} & \hspace{-0.4cm} $\begin{array}{c}
0110011 \\
1010101 \\
0000000 \\
1011010 \\
1101001 \\
0111100
\end{array}$ \hspace{-0.4cm} & \hspace{-0.4cm} $\begin{array}{c}
0110011 \\
1010101 \\
0001111 \\
1010101 \\
1100110 \\
0000000
\end{array}$ \hspace{-0.4cm} & \hspace{-0.4cm} $\begin{array}{c}
0110011 \\
1010101 \\
0001111 \\
1010101 \\
1100110 \\
0001111
\end{array}$ \hspace{-0.4cm} & \hspace{-0.4cm} $\begin{array}{c}
0110011 \\
1010101 \\
0001111 \\
1010101 \\
1101001 \\
0000000
\end{array}$ \hspace{-0.4cm} & \hspace{-0.4cm} $\begin{array}{c}
0110011 \\
1010101 \\
0001111 \\
1010101 \\
1101001 \\
0001111
\end{array}$ \hspace{-0.4cm} & \hspace{-0.4cm} $\begin{array}{c}
0110011 \\
1010101 \\
0001111 \\
1010101 \\
1010101 \\
1100110
\end{array}$ \hspace{-0.4cm} & \hspace{-0.4cm} $\begin{array}{c}
0110011 \\
1010101 \\
0001111 \\
1010101 \\
1010101 \\
1101001
\end{array}$ \hspace{-0.4cm} & \hspace{-0.4cm} $\begin{array}{c}
0110011 \\
1010101 \\
0001111 \\
1010101 \\
1011010 \\
1100110
\end{array}$ \hspace{-0.4cm} & \hspace{-0.4cm} $\begin{array}{c}
0110011 \\
1010101 \\
0001111 \\
1010101 \\
1011010 \\
1101001
\end{array}$ \hspace{-0.4cm}
\\ \hline
\hspace{-0.4cm} $\begin{array}{c}
0110011 \\
1010101 \\
0001111 \\
1010101 \\
0110011 \\
0110011
\end{array}$ \hspace{-0.4cm} & \hspace{-0.4cm} $\begin{array}{c}
0110011 \\
1010101 \\
0001111 \\
1010101 \\
0110011 \\
0111100
\end{array}$ \hspace{-0.4cm} & \hspace{-0.4cm} $\begin{array}{c}
0110011 \\
1010101 \\
0001111 \\
1010101 \\
0111100 \\
0110011
\end{array}$ \hspace{-0.4cm} & \hspace{-0.4cm} $\begin{array}{c}
0110011 \\
1010101 \\
0001111 \\
1010101 \\
0111100 \\
0111100
\end{array}$ \hspace{-0.4cm} & \hspace{-0.4cm} $\begin{array}{c}
0110011 \\
1010101 \\
0001111 \\
1010101 \\
0000000 \\
1010101
\end{array}$ \hspace{-0.4cm} &  \hspace{-0.4cm} $\begin{array}{c}
0110011 \\
1010101 \\
0001111 \\
1010101 \\
0000000 \\
1011010
\end{array}$ \hspace{-0.4cm} &  \hspace{-0.4cm} $\begin{array}{c}
0110011 \\
1010101 \\
0001111 \\
1010101 \\
0001111 \\
1010101
\end{array}$ \hspace{-0.4cm} &  \hspace{-0.4cm} $\begin{array}{c}
0110011 \\
1010101 \\
0001111 \\
1010101 \\
0001111 \\
1011010
\end{array}$ \hspace{-0.4cm} &  \hspace{-0.4cm} $\begin{array}{c}
0110011 \\
1010101 \\
0001111 \\
1011010 \\
1100110 \\
1100110
\end{array}$ \hspace{-0.4cm} &  \hspace{-0.4cm} $\begin{array}{c}
0110011 \\
1010101 \\
0001111 \\
1011010 \\
1100110 \\
1101001
\end{array}$ \hspace{-0.4cm} &  \hspace{-0.4cm} $\begin{array}{c}
0110011 \\
1010101 \\
0001111 \\
1011010 \\
1101001 \\
1100110
\end{array}$ \hspace{-0.4cm} &  \hspace{-0.4cm} $\begin{array}{c}
0110011 \\
1010101 \\
0001111 \\
1011010 \\
1101001 \\
1101001
\end{array}$ \hspace{-0.4cm}
\\ \hline
\hspace{-0.4cm} $\begin{array}{c}
0110011 \\
1010101 \\
0001111 \\
1011010 \\
1010101 \\
0000000
\end{array}$ \hspace{-0.4cm} & \hspace{-0.4cm} $\begin{array}{c}
0110011 \\
1010101 \\
0001111 \\
1011010 \\
1010101 \\
0001111
\end{array}$ \hspace{-0.4cm} & \hspace{-0.4cm} $\begin{array}{c}
0110011 \\
1010101 \\
0001111 \\
1011010 \\
1011010 \\
0000000
\end{array}$ \hspace{-0.4cm} & \hspace{-0.4cm} $\begin{array}{c}
0110011 \\
1010101 \\
0001111 \\
1011010 \\
1011010 \\
0001111
\end{array}$ \hspace{-0.4cm} & \hspace{-0.4cm} $\begin{array}{c}
0110011 \\
1010101 \\
0001111 \\
1011010 \\
0110011 \\
1010101
\end{array}$ \hspace{-0.4cm} & \hspace{-0.4cm} $\begin{array}{c}
0110011 \\
1010101 \\
0001111 \\
1011010 \\
0110011 \\
1011010
\end{array}$ \hspace{-0.4cm} & \hspace{-0.4cm} $\begin{array}{c}
0110011 \\
1010101 \\
0001111 \\
1011010 \\
0111100 \\
1010101
\end{array}$ \hspace{-0.4cm} & \hspace{-0.4cm} $\begin{array}{c}
0110011 \\
1010101 \\
0001111 \\
1011010 \\
0111100 \\
1011010
\end{array}$ \hspace{-0.4cm} & \hspace{-0.4cm} $\begin{array}{c}
0110011 \\
1010101 \\
0001111 \\
1011010 \\
0000000 \\
0110011
\end{array}$ \hspace{-0.4cm} & \hspace{-0.4cm} $\begin{array}{c}
0110011 \\
1010101 \\
0001111 \\
1011010 \\
0000000 \\
0111100
\end{array}$ \hspace{-0.4cm} & \hspace{-0.4cm} $\begin{array}{c}
0110011 \\
1010101 \\
0001111 \\
1011010 \\
0001111 \\
0110011
\end{array}$ \hspace{-0.4cm} & \hspace{-0.4cm} $\begin{array}{c}
0110011 \\
1010101 \\
0001111 \\
1011010 \\
0001111 \\
0111100
\end{array}$ \hspace{-0.4cm}
\\ \hline
\end{tabular}
\end{tiny}
\end{table}

\begin{table}[ht]
\centering \caption{Extensions of Type D with $Y_1=\cY_{13}$, $Y_2=\cY_{14}$, and $Y_3=\cY_{15}$}
\label{tab:S_D6}
\begin{tiny}
\begin{tabular}{|c|c|c|c|c|c|c|c|c|c|c|c|}
\hline
\hspace{-0.4cm} $\begin{array}{c}
0110011 \\
0000000 \\
1010101 \\
1101001 \\
0000000 \\
1100110
\end{array}$ \hspace{-0.4cm} & \hspace{-0.4cm} $\begin{array}{c}
0110011 \\
0000000 \\
1010101 \\
1101001 \\
0000000 \\
1101001
\end{array}$ \hspace{-0.4cm} & \hspace{-0.4cm} $\begin{array}{c}
0110011 \\
0000000 \\
1010101 \\
1101001 \\
0001111 \\
1100110
\end{array}$ \hspace{-0.4cm} & \hspace{-0.4cm} $\begin{array}{c}
0110011 \\
0000000 \\
1010101 \\
1101001 \\
0001111 \\
1101001
\end{array}$ \hspace{-0.4cm} & \hspace{-0.4cm} $\begin{array}{c}
0110011 \\
0000000 \\
1010101 \\
1101001 \\
0110011 \\
0000000
\end{array}$ \hspace{-0.4cm} & \hspace{-0.4cm} $\begin{array}{c}
0110011 \\
0000000 \\
1010101 \\
1101001 \\
0110011 \\
0001111
\end{array}$ \hspace{-0.4cm} & \hspace{-0.4cm} $\begin{array}{c}
0110011 \\
0000000 \\
1010101 \\
1101001 \\
0111100 \\
0000000
\end{array}$ \hspace{-0.4cm} & \hspace{-0.4cm} $\begin{array}{c}
0110011 \\
0000000 \\
1010101 \\
1101001 \\
0111100 \\
0001111
\end{array}$ \hspace{-0.4cm} & \hspace{-0.4cm} $\begin{array}{c}
0110011 \\
0000000 \\
1010101 \\
1101001 \\
1010101 \\
1010101
\end{array}$ \hspace{-0.4cm} & \hspace{-0.4cm} $\begin{array}{c}
0110011 \\
0000000 \\
1010101 \\
1101001 \\
1010101 \\
1011010
\end{array}$ \hspace{-0.4cm} & \hspace{-0.4cm} $\begin{array}{c}
0110011 \\
0000000 \\
1010101 \\
1101001 \\
1011010 \\
1010101
\end{array}$ \hspace{-0.4cm} & \hspace{-0.4cm} $\begin{array}{c}
0110011 \\
0000000 \\
1010101 \\
1101001 \\
1011010 \\
1011010
\end{array}$ \hspace{-0.4cm}
\\ \hline
\hspace{-0.4cm} $\begin{array}{c}
0110011 \\
0000000 \\
1010101 \\
1101001 \\
1100110 \\
0110011
\end{array}$ \hspace{-0.4cm} & \hspace{-0.4cm} $\begin{array}{c}
0110011 \\
0000000 \\
1010101 \\
1101001 \\
1100110 \\
0111100
\end{array}$ \hspace{-0.4cm} & \hspace{-0.4cm} $\begin{array}{c}
0110011 \\
0000000 \\
1010101 \\
1101001 \\
1101001 \\
0110011
\end{array}$ \hspace{-0.4cm} & \hspace{-0.4cm} $\begin{array}{c}
0110011 \\
0000000 \\
1010101 \\
1101001 \\
1101001 \\
0111100
\end{array}$ \hspace{-0.4cm} & \hspace{-0.4cm} $\begin{array}{c}
0110011 \\
0001111 \\
1010101 \\
1100110 \\
1100110 \\
0000000
\end{array}$ \hspace{-0.4cm} & \hspace{-0.4cm} $\begin{array}{c}
0110011 \\
0001111 \\
1010101 \\
1100110 \\
1100110 \\
0001111
\end{array}$ \hspace{-0.4cm} & \hspace{-0.4cm} $\begin{array}{c}
0110011 \\
0001111 \\
1010101 \\
1100110 \\
1101001 \\
0000000
\end{array}$ \hspace{-0.4cm} & \hspace{-0.4cm} $\begin{array}{c}
0110011 \\
0001111 \\
1010101 \\
1100110 \\
1101001 \\
0001111
\end{array}$ \hspace{-0.4cm} & \hspace{-0.4cm} $\begin{array}{c}
0110011 \\
0001111 \\
1010101 \\
1100110 \\
1010101 \\
1100110
\end{array}$ \hspace{-0.4cm} & \hspace{-0.4cm} $\begin{array}{c}
0110011 \\
0001111 \\
1010101 \\
1100110 \\
1010101 \\
1101001
\end{array}$ \hspace{-0.4cm} & \hspace{-0.4cm} $\begin{array}{c}
0110011 \\
0001111 \\
1010101 \\
1100110 \\
1011010 \\
1100110
\end{array}$ \hspace{-0.4cm} & \hspace{-0.4cm} $\begin{array}{c}
0110011 \\
0001111 \\
1010101 \\
1100110 \\
1011010 \\
1101001
\end{array}$ \hspace{-0.4cm}
\\ \hline
\hspace{-0.4cm} $\begin{array}{c}
0110011 \\
0001111 \\
1010101 \\
1100110 \\
0110011 \\
0110011
\end{array}$ \hspace{-0.4cm} & \hspace{-0.4cm} $\begin{array}{c}
0110011 \\
0001111 \\
1010101 \\
1100110 \\
0110011 \\
0111100
\end{array}$ \hspace{-0.4cm} & \hspace{-0.4cm} $\begin{array}{c}
0110011 \\
0001111 \\
1010101 \\
1100110 \\
0111100 \\
0110011
\end{array}$ \hspace{-0.4cm} & \hspace{-0.4cm} $\begin{array}{c}
0110011 \\
0001111 \\
1010101 \\
1100110 \\
0111100 \\
0111100
\end{array}$ \hspace{-0.4cm} & \hspace{-0.4cm} $\begin{array}{c}
0110011 \\
0001111 \\
1010101 \\
1100110 \\
0000000 \\
1010101
\end{array}$ \hspace{-0.4cm} & \hspace{-0.4cm} $\begin{array}{c}
0110011 \\
0001111 \\
1010101 \\
1100110 \\
0000000 \\
1011010
\end{array}$ \hspace{-0.4cm} & \hspace{-0.4cm} $\begin{array}{c}
0110011 \\
0001111 \\
1010101 \\
1100110 \\
0001111 \\
1010101
\end{array}$ \hspace{-0.4cm} & \hspace{-0.4cm} $\begin{array}{c}
0110011 \\
0001111 \\
1010101 \\
1100110 \\
0001111 \\
1011010
\end{array}$ \hspace{-0.4cm} & \hspace{-0.4cm} $\begin{array}{c}
0110011 \\
0001111 \\
1010101 \\
1101001 \\
1100110 \\
1100110
\end{array}$ \hspace{-0.4cm} & \hspace{-0.4cm} $\begin{array}{c}
0110011 \\
0001111 \\
1010101 \\
1101001 \\
1100110 \\
1101001
\end{array}$ \hspace{-0.4cm} & \hspace{-0.4cm} $\begin{array}{c}
0110011 \\
0001111 \\
1010101 \\
1101001 \\
1101001 \\
1100110
\end{array}$ \hspace{-0.4cm} & \hspace{-0.4cm} $\begin{array}{c}
0110011 \\
0001111 \\
1010101 \\
1101001 \\
1101001 \\
1101001
\end{array}$ \hspace{-0.4cm}
\\ \hline
\hspace{-0.4cm} $\begin{array}{c}
0110011 \\
0001111 \\
1010101 \\
1101001 \\
1010101 \\
0000000
\end{array}$ \hspace{-0.4cm} & \hspace{-0.4cm} $\begin{array}{c}
0110011 \\
0001111 \\
1010101 \\
1101001 \\
1010101 \\
0001111
\end{array}$ \hspace{-0.4cm} & \hspace{-0.4cm} $\begin{array}{c}
0110011 \\
0001111 \\
1010101 \\
1101001 \\
1011010 \\
0000000
\end{array}$ \hspace{-0.4cm} & \hspace{-0.4cm} $\begin{array}{c}
0110011 \\
0001111 \\
1010101 \\
1101001 \\
1011010 \\
0001111
\end{array}$ \hspace{-0.4cm} & \hspace{-0.4cm} $\begin{array}{c}
0110011 \\
0001111 \\
1010101 \\
1101001 \\
0110011 \\
1010101
\end{array}$ \hspace{-0.4cm} & \hspace{-0.4cm} $\begin{array}{c}
0110011 \\
0001111 \\
1010101 \\
1101001 \\
0110011 \\
1011010
\end{array}$ \hspace{-0.4cm} & \hspace{-0.4cm} $\begin{array}{c}
0110011 \\
0001111 \\
1010101 \\
1101001 \\
0111100 \\
1010101
\end{array}$ \hspace{-0.4cm} & \hspace{-0.4cm} $\begin{array}{c}
0110011 \\
0001111 \\
1010101 \\
1101001 \\
0111100 \\
1011010
\end{array}$ \hspace{-0.4cm} & \hspace{-0.4cm} $\begin{array}{c}
0110011 \\
0001111 \\
1010101 \\
1101001 \\
0000000 \\
0110011
\end{array}$ \hspace{-0.4cm} & \hspace{-0.4cm} $\begin{array}{c}
0110011 \\
0001111 \\
1010101 \\
1101001 \\
0000000 \\
0111100
\end{array}$ \hspace{-0.4cm} & \hspace{-0.4cm} $\begin{array}{c}
0110011 \\
0001111 \\
1010101 \\
1101001 \\
0001111 \\
0110011
\end{array}$ \hspace{-0.4cm} & \hspace{-0.4cm} $\begin{array}{c}
0110011 \\
0001111 \\
1010101 \\
1101001 \\
0001111 \\
0111100
\end{array}$ \hspace{-0.4cm}
\\ \hline
\end{tabular}
\end{tiny}
\end{table}
\clearpage

\end{appendix}

%%%%%%%%%%%%%%%%%%%%%%%%%%%%%%%%%%%%%%%%%%%%%%%%%%%%%%%%%%%%%%%%%%%%%%
%%%%%%%%%%%%%%%%%%%%%%%%%%%%%%%%%%%%%%%%%%%%%%%%%%%%%%%%%%%%%%%%%%%%%%
%%%%%%%%%%%%%%%%%%%%%%%%%%%%%%%%%%%%%%%%%%%%%%%%%%%%%%%%%%%%%%%%%%%%%%
%%%%%%%%%%%%%%%%%%%%%%%%%%%%%%%%%%%%%%%%%%%%%%%%%%%%%%%%%%%%%%%%%%%%%%


\begin{thebibliography}{99}
\bibitem{AAK01}
     {\sc R. Ahlswede, H. K. Aydinian, and L. H. Khachatrian,}
     {\sl On perfect codes and related concepts,}
     {\em Designs, Codes, and Cryptography,} 22 (2001), 221--237.
\bibitem{BJL}
    {\sc T. Beth, D. Jungnickel, and H. Lenz},
    \emph{Design Theory, Volume I}, 2nd ed.,
    Cambridge Univ.\ Press, Cambridge, 1999.
    {\sc T. Beth, D. Jungnickel, and H. lenz,}
    {\em Design Theory, Volume I}, 2nd ed.
    Cambridge University Press, Cambridge 1999.
\bibitem{Beu74}
     {\sc A. Beutelspacher,}
     {\sl On parallelisms in finite projective spaces,}
     {\em Geometriae Dedicata,} 3 (1974), 35--40.
\bibitem{Beu78}
    {\sc A. Beutelspacher},
    {\sl Parallelismen in unendlichen projektiven Raumen endlicher Dimension},
    {\em Geom.\ Dedicata} 7 (1978) 499--506.
\bibitem{BEOVW}
     {\sc M. Braun, T. Etzion, P. R. J. \"Osterg\aa rd, A. Vardy, and A. Wassermann}
     {\sl Existence of $q$-Analogs of Steiner Systems,}
     {\em Forum of Mathematics, Pi}, 4 (2016), 1--14.
\bibitem{BKL}
    {\sc M. Braun, A. Kerber and R. Laue,}
    {\sl Systematic construction of $q$-analogs of $t$-$(v,k,\lambda)$-designs,}
    {\em Designs, Codes, and Cryptography} 34 (2005) 55--70.
\bibitem{BKN15}
     {\sc M. Braun, M. Kiermaier, and A. Naki\'{c}}
     {\sl On the automorphism group of a binary $q$-analog of the Fano plane,}
    {\em European Journal of Combinatorics,} 51 (2016), 443--457.
\bibitem{C1850}
    {\sc A. Cayley},
    {\sl On the triadic arrangements of seven and fifteen things},
    {\em Philosophical Mag.} 37 (1850) 50--53.
\bibitem{Cam74}
    P.\,{\sc Cameron},
    Generalisation of Fisher's inequality to fields with more than one element,
    in T.P.\,{\sc McDonough} and V.C.\,{\sc Mavron}, Eds.,
		\emph{Combinatorics},
    London Math.\ Soc.\ Lecture Note Ser.\ 13
    Cambridge Univ.\ Press, Cambridge,		
    1974, pp.~9--13.
\bibitem{Cam74a}
    {\sc P. Cameron},
    {\sl Locally symmetric designs},
    {\em Geometriae Dedicata,} 3 (1974), 65--76.
\bibitem{Cohn}
    {\sc H. Cohn},
    {\sl Projective geometry over $\mathbb{F}_1$ and the Gaussian binomial coefficients,}
    {\em Amer.\ Math.\ Monthly} 111 (2004) 487--495.
\bibitem{CoDi07}
    C. J. Colbourn and J. H. Dinitz,
    {\em Handbook of Combinatorial Designs},
    Boca Raton, Florida: Chapman and Hall/CRC, 2007.
\bibitem{Del76}
    {\sc P. Delsarte},
    {\sl Association schemes and $t$-designs in regular semilattices},
    {\em Journal of Combinatorial Theory, Series A,} 20 (1976), 230--243.
\bibitem{DePh17}
    {\sc E. Dempsey and K.evin T. Phelps},
    {\sl On 2-Steiner systems $S_2[2,3,n]$},
    {\em preprint}.
%\bibitem{Etz14}
%    {\sc T. Etzion,}
%    {\sl Covering of subspaces by subspaces,}
%    {\em Designs, Codes, and Cryptography,} 72 (2014), 405--421.
\bibitem{Etz15}
    {\sc T. Etzion,}
    {\sl A new approach to examine $q$-Steiner systems,}
     {\em arxiv.org/abs/1507.08503}, July 2015.
\bibitem{EtSt16}
  {\sc T. Etzion and L. Storme,}
  {\sl Galois geometries and coding theory,}
  {\em Designs, Codes, and Cryptography} 78 (2016) 311--350.
\bibitem{EtVa11}
     {\sc T. Etzion and A. Vardy,}
     {\sl Error-correcting codes in projective space,}
     {\em IEEE Trans. on Inform. Theory} 57 (2011) 1165--1173.
\bibitem{EtVa11a}
     {\sc T. Etzion and A. Vardy,}
     {\sl On $q$-analogs for Steiner systems and covering designs,}
     {\em Advances in Mathematics of Communications}, 5 (2011), 161--176.
\bibitem{EtZh17}
     {\sc T. Etzion and H. Zhang,}
     {\sl New subspace codes and designs for network coding,}
     {\em in preparation}.
\bibitem{Euler}
    {\sc L. Euler},
    {\em Consideratio quarumdam serierum quae singularibus
              proprietatibus sunt praeditae},
    Novi Commentarii Academiae Scientiarum Petropolitanae 3
    (1750--1751) pp. 10--12, 86--108;
    Opera Omnia, Ser.\,I, vol.\,14, B.G.\,Teubner, Leipzig, 1925, pp.~516--541.
\bibitem{FLV14}
    {\sc A. Fazeli, S. Lovett, and A. Vardy,}
    {\sl Nontrivial $t$-designs over finite fields exist for all $t$},
    {\em J.\ Combin.\ Theory Ser.\,A} 127 (2014) 149--160.
\bibitem{GKLO16}
     {\sc S. Glock, D. K\"{u}hn, A. Lo, and D. Osthus,}
     {\sl The existence of designs via iterative absorption,}
     {\em arxiv.org/abs/1611.06827}, November 2016.
\bibitem{GR}
    {\sc J. R. Goldman and G. C. Rota},
    {\sl On the foundations of combinatorial theory IV: Finite vector spaces and
		Eulerian generating functions},
    {\em Stud.\ Appl.\ Math.} 49 (1970)	239--258.
\bibitem{Kee14}
     {\sc P. Keevash,}
     {\sl The existence of designs,}
     {\em arxiv.org/abs/1401.3665, January 2014}.
\bibitem{KiLa15}
     {\sc M. Kiermaier and R. Laue,}
     {\sl Derived and residual subspace designs,}
     {\em Advances in Mathematics of Communications}, 9 (2015), 105--115.
%\bibitem{KiPa15}
%     {\sc M. Kiermaier and M. O. Pav\v{c}evi\'{c},}
%     {\sl Imntersection numbers for subspace designs,}
%     {\em Journal of Combinatorial Designs}, 23 (2015), 463--480.
\bibitem{K1847}
    {\sc T. P. Kirkman},
    {\sl On a problem in combinations},
    {\em Cambridge \& Dublin Math.\ J.} II (1847) 191--204.
\bibitem{KvA09}
    {\sc E. Koelink and W. van Assche},
    {\sl Leonhard Euler and a $q$-analogue of the logarithm},
    {\em Proc.\ Amer.\ Math.\ Soc.} 137 (2009) 1663--1676.
\bibitem{KoKs08}
    {\sc R. Koetter and F. R. Kschischang,}
    {\sl Coding for errors and erasures in random network coding,}
    {\em IEEE Trans. on Inform. Theory},  54 (2008), 3579--3591.
\bibitem{Met99}
    {\sc K. Metsch},
	{\sl Bose-Burton type theorems for finite projective, affine and polar spaces,}
		in J.D.\,{\sc Lamb} and D.A.\,{\sc Preece}, Eds.,
    {\em Surveys in Combinatorics, 1999},
                London Math.\ Soc.\ Lecture Note Ser.\ 267,
		Cambridge Univ.\ Press, Cambridge, pp.~137--166, 1999.
\bibitem{MMY95}
    {\sc M. Miyakawa, A. Munemasa and S. Yoshiara,}
    {\sl On a class of small\/ $2$-designs over ${\rm GF}(q)$,}
    {\em Journal Combinatorial Designs} 3 (1995) 61--77.
\bibitem{P1835}
    {\sc J. Pl\"ucker},
    {\em System der analytischen Geometrie: auf neue Betrachtungsweisen
    gegr\"undet, und insbesondere eine ausf\"uhrliche Theorie der Curven
    dritter Ordnung enthaltend},
    Duncker und Humblot, Berlin, 1835.
\bibitem{RaSi89}
    {\sc D. K. Ray-Chaudhuri and N. M. Singhi},
    {\sl $q$-analogues of $t$-designs and their existence,}
    {\em Linear Algebra Appl.} 114/115 (1989) 57--68.
\bibitem{ScEt02}
    {\sc M. Schwartz and T. Etzion,}
    {\sl Codes and anticodes in the Grassman graph,}
    {\em Journal Combinatorial Theory, Series A,} 97 (2002), 27--42.
\bibitem{Spe68}
    {\sc J. Spencer,}
    {\sl Maximal consistent families of triples,}
    {\em Journal Combinatorial Theory,} 5 (1968), 1--8.
\bibitem{S1853}
    {\sc J. Steiner},
    {\sl Combinatorische Aufgabe,}
    {\em J.\ Reine Angewandte Math.} 45 (1853) 181--182.
\bibitem{Suz90}
    {\sc H. Suzuki,}
    {\sl $2$-designs over ${\rm GF}(2^m)$,}
    {\em Graphs Combin.} 6 (1990) 293--296.
\bibitem{Suz90a}
    {\sc H. Suzuki,}
    {\sl On the inequalities of $t$-designs over a finite field,}
    {\em European Journal Combin.} 11 (1990) 601--607.
\bibitem{Suz92}
    {\sc H. Suzuki,}
    {\sl $2$-designs over ${\rm GF}(q)$,}
    {\em Graphs Combin.} 8 (1992) 381--389.
\bibitem{Tho87}
    {\sc S. Thomas,}
    {\sl Designs over finite fields,}
    {\em Geom. Dedicata} 21 (1987) 237--242.
\bibitem{Tho96}
    {\sc S. Thomas,}
    {\sl Designs and partial geometries over finite fields,}
    {\em Geometriae Dedicata,} 63 (1996), 247--253.
\bibitem{Tits57}
    {\sc J. Tits},
    {\sl Sur les analogues alg\'ebriques des groupes semi-simples complexes},
    in {\em Colloque d'Alg\'ebre Sup\`erieure},
    tenu \'a Bruxelles du 19 au 22 d\'ecembre 1956,
    Centre Belge de Recherches Math\'ematiques \'Etablissements Ceuterick,
    Louvain, Paris: Librairie Gauthier-Villars, 1957, pp.~261--289.
\bibitem{vLWi92}
    {\sc J. H. van Lint and R. M. Wilson,}
    {\em A course in Combinatorics},
    Cambridge University Press, 1992.
\bibitem{Wang}
    {\sc J. Wang},
    {\sl Quotient sets and subset-subspace analogy,}
    {\em Adv.\ Appl.\ Math.} 23 (1999) 333--339.
\end{thebibliography}
\end{document}